\documentclass[pdflatex,sn-mathphys-num]{sn-jnl}

\usepackage{graphicx}%
\usepackage{multirow}%
\usepackage{amsmath,amssymb,amsfonts}%
\usepackage{amsthm}%
\usepackage{mathrsfs}%
\usepackage{mathtools}%
\usepackage[title]{appendix}%
\usepackage{xcolor}%
\usepackage{textcomp}%
\usepackage{manyfoot}%
\usepackage{booktabs}%
\usepackage{algorithm}%
\usepackage{algorithmicx}%
\usepackage{algpseudocode}%
\usepackage{listings}%
\usepackage{tikz-cd}
\usepackage{tikz, pgfplots}%
\usetikzlibrary{positioning}%
\usetikzlibrary{decorations.markings}%
\pgfplotsset{compat=1.17}%
\usepackage{xcolor}
\usepackage[USenglish]{babel}
\usepackage[utf8]{inputenc}
\usepackage{comment}
\usepackage{blkarray}

\numberwithin{equation}{section}
\newcounter{dummy}\numberwithin{dummy}{section}
\theoremstyle{thmstyleone}%
\newtheorem{thm}[dummy]{Theorem}
\newtheorem*{Thm}{Main Theorem}
\newtheorem{prop}[dummy]{Proposition}%
\newtheorem{lemma}[dummy]{Lemma}
\newtheorem{cor}[dummy]{Corollary}

\theoremstyle{thmstyletwo}%
\newtheorem{example}[dummy]{Example}%
\newtheorem{rmk}[dummy]{Remark}%

\theoremstyle{thmstylethree}%
\newtheorem{defin}[dummy]{Definition}%

\DeclareMathOperator{\ind}{ind}
\DeclareMathOperator{\tr}{tr}

\newcommand\smallO{
  \mathchoice
    {{\scriptstyle\mathcal{O}}}
    {{\scriptstyle\mathcal{O}}}
    {{\scriptscriptstyle\mathcal{O}}}
    {\scalebox{.7}{$\scriptscriptstyle\mathcal{O}$}}
  }

\begin{document}

\title[Article Title]{Local index theory for geometric first-order differential operators}


\author{\fnm{Alberto} \sur{Richtsfeld}}\email{alberto.richtsfeld@math.su.se}

\affil{\orgdiv{Matematiska institutionen}, \orgname{Stockholms Universitet}, \orgaddress{\city{Stockholm}, \postcode{106 91}, \country{Sweden}}}


\abstract{We introduce the concept of chiral geometric operators and use Gilkey's invariance theory to prove the local index theorem for these operators.
	In other words, we demonstrate that the supertrace of the heat kernel of a given geometric operator converges as time approaches zero and that this limit is the Chern--Weil form of the Atiyah--Singer integrand. In addition to classical Dirac-type operators that appear in geometry, chiral geometric operators include all higher Dirac operators. This includes in particular the Rarita--Schwinger operator. We also construct a new class of such operators on four-manifolds called higher signature operators.}




\maketitle

\section{Introduction}\label{sec1}

The local index theorem, a refined form of the Atiyah--Singer index theorem, examines the short-time behavior of the heat kernel supertrace associated to the square of a first-order differential operator. The classical theorem asserts that, on a Riemannian (spin) manifold $(M,g)$, the heat kernel supertrace of the Laplacian for the classical geometric elliptic complexes converges to a limit expressed as polynomials in the Pontryagin forms associated with the metric $g$. Patodi first proved this for the de Rham complex \cite{patodiCurvatureEigenformsLaplace1971}, while Gilkey extended the result to the signature, spin and Dolbeault complex via the development of a theory for differential geometric invariants \cite{gilkeyCurvatureEigenvaluesLaplacian1973}. Atiyah, Bott, and Patodi later simplified the exposition of this proof \cite{atiyahHeatEquationIndex1973,atiyahErrataPaperHeat1975}. Witten's \cite{wittenConstraintsSupersymmetryBreaking1982} insights from supersymmetric quantum field theory connected the index of a Dirac-type operator with the fermion number, leading Alvarez-Gaumé \cite{alvarez-gaumeSupersymmetryAtiyahSingerIndex1983} to outline a novel proof for the local index theorem. Subsequently, Getzler provided a purely analytic proof, while Bismut \cite{bismutAtiyahSingerTheorems1984} contributed a probabilistic approach. Berline and Vergne \cite{berlineComputationEquivariantIndex1985} offered a geometric perspective through their analysis of the Laplacian's heat kernel expansion on the principal spin bundle. These latter proofs have replaced Gilkey's proof in popularity, but they make heavy use of the fact that the square of a Dirac-type operator is Laplace-type, whereas Gilkey's proof merely uses this fact to simplify calculations. 

In the following, we are interested in the local index theorem for geometric operators that are not of Dirac-type, with a particular focus on the Rarita--Schwinger operator. Introduced by Rarita and Schwinger \cite{raritaTheoryParticlesHalfIntegral1941}  the Rarita--Schwinger operator is used to describe the wave equations of spin-\(\frac{3}{2}\)-particles, and plays an important role in the study of supergravity and superstring theory. In geometry, however, the Rarita--Schwinger operator is rarely studied. Semmelmann studied the spectrum of the Rarita--Schwinger operator on the sphere and complex projective space \cite{semmelmannKomplexeKontaktstrukturenUnd1995}. Bure\v{s} \cite{buresHigherSpinDirac1999} introduced it from the point of view of representation theory, showing that the Dirac and Rarita--Schwinger operators appear as the first two terms in a sequence of elliptic first-order differential operators associated with representations of the spin group, called higher Dirac operators. In particular, Bure\v{s} computes their index on closed manifolds. Bär and Bandara \cite{barBoundaryValueProblems2022} studied boundary conditions for general first-order differential operators and in particular established the Fredholm property for the APS boundary condition. Branson and Hijazi \cite{bransonBochnerWeitzenbockFormulas2002} derived a Weitzenböck formula for the Rarita--Schwinger operator, whose zero-order term has no particular sign under usual geometric hypotheses, suggesting that the vanishing of the index has no direct connection with geometry. Homma and Semmelmann \cite{hommaKernelRaritaSchwinger2019} studied the problem of counting Rarita--Schwinger fields, classifying all positive quaternion Kähler manifolds and spin symmetric spaces. The case of negative Kähler Einstein manifolds is treated by Bär and Mazzeo \cite{barManifoldsManyRarita2021}, where they find a sequence of manifolds whose number of Rarita--Schwinger fields diverges to infinity. Recently, Nguyen \cite{nguyenPin2EquivariancePropertyRarita2023} established Seiberg-Witten equations for the Rarita--Schwinger operator and found a topological condition ensuring the existence of solutions in dimension four. Haj Saeedi Sadegh and Nguyen \cite{hajsaeedisadeghThreedimensionalSeibergWitten} extend the study of the Rarita--Schwinger-Seiberg-Witten equations to the three-dimensional case and prove a compactness result for the moduli space of solutions to these equations. 

The idea of Gilkey's proof is as follows: Since the operators in question are given more or less canonically by the Riemannian metric on the manifold, the formulas for the heat coefficient traces depend locally only on the coefficients $g_{ij}$ of the metric and their derivatives.  The formulas are independent of the chosen coordinate system and are therefore $\mathrm{O}(n)$-invariants. This allows the use of Weyl's fundamental theorem of invariant theory. Using the scaling properties of the heat coefficients, one can deduce that the supertrace of the heat kernel converges to a limit which can be expressed as a polynomial in Pontryagin forms. The fact that the Pontryagin classes are a cobasis of the oriented cobordism ring allows the limit to be determined as the integrand given by the Atiyah--Singer index theorem.

We use the same strategy to obtain the local index theorem for the Rarita--Schwinger operator and higher Dirac operators. The goal of this paper is to build the most general case for which Gilkey's proof works. To this end, we introduce the notion of chiral geometric operators, a class of operators that includes the Dirac operator, the Rarita--Schwinger operator, and the higher Dirac operators. We also construct a new family of examples of chiral geometric operators, so called higher signature operators. This family of operators contains the signature operator as its simplest member. Using Atiyah, Bott and Patodi's version of Gilkey's theorem \cite{atiyahHeatEquationIndex1973}, we then reformulate the proof of the local index theorem to hold for chiral geometric operators. In doing so, we unify the approaches of Atiyah, Bott, Patodi \cite{atiyahHeatEquationIndex1973,atiyahErrataPaperHeat1975} and Gilkey \cite{gilkeyInvarianceTheoryHeat1994,gilkeyCurvatureEigenvaluesLaplacian1973}, trying to make the proof as simple as possible. The result is a concise proof of the local index theorem for chiral geometric operators. Therefore, this paper can be also interesting to readers who want to understand Gilkey's method.

We give a short summary of this article:
Let $(M^n,g)$ be an even-dimensional closed Riemannian spin manifold, the Rarita--Schwinger operator
$$Q:C^\infty(M,\Sigma^{\frac{3}{2}}M)\rightarrow C^\infty(M,\Sigma^{\frac{3}{2}}M)$$
is a self-adjoint, elliptic first-order differential operator acting on sections of the $3/2$-spinor bundle $\Sigma^{\frac{3}{2}}M$ (for a definition, see Remark \ref{rmk:higher_Dirac_operator}). The orientation of $M$ gives a splitting 
$$\Sigma^{\frac{3}{2}}M = \Sigma^{\frac{3}{2}}_+M\oplus \Sigma^{\frac{3}{2}}_-M,$$
such that when the orientation is reversed, the roles of $\Sigma^{\frac{3}{2}}_\pm M$ intertwine. The Rarita--Schwinger operator is odd with respect to this splitting, i.e.
$$ Q = \begin{pmatrix}
	0 & Q_-\\
	Q_+ & 0
\end{pmatrix}.$$
The heat semigroup $\{e^{-tQ^2}, t>0\}$ of $Q$ consists of smoothing operators acting on sections of $\Sigma^{3/2}M$ and hence their Schwartz kernels $k_t$ are smooth.
These are the requirements we need in order to make the Gilkey proof work and are met by the Rarita--Schwinger operator: The operator must be
\begin{itemize}
	\item determined by the geometry of the manifold,
	\item self-adjoint, and
	\item odd with respect to a $\mathbb{Z}_2$-grading on the bundle determined by the orientation of the manifold.
\end{itemize}
We summarize these properties in Section \ref{sec:geom_operators} in the notion of chiral geometric operators (see Definition \ref{def:chiralgeomsymbol}). Apart from the higher Dirac operators, we find a new class of chiral geometric operators. This class has the signature operator on a $4$-dimensional manifold as its simplest member, we thus name these operators higher signature operators.
By wrapping Gilkey's proof into a new formalism, we then obtain the local index theorem for chiral geometric operators:

\begin{Thm}
	Assume the notation introduced in Section \ref{sec:geom_operators} and \ref{section:localindex}.
	Let $(\sigma,\varepsilon)$ be a chiral $G_n$-geometric symbol for $n=2l$ even. Let $(M,P)$ be a Riemannian $G_n$-manifold equipped with an hermitian vector bundle $\xi$. Let $D^{\xi}$ be the induced twisted geometric operator. Then the equality
	\begin{equation*}
		\lim_{t\searrow 0}\left(\operatorname{str}\exp(-t (D^{\xi})^{2})\cdot\mathrm{dvol}_g\right) = (-1)^{l} \left(\mathrm{ch}(\nabla^{\xi}) \cdot \frac{\mathrm{ch}(V_{+})-\mathrm{ch}(V_{-})}{\chi}(\nabla^{g})\cdot \hat{A}(\nabla^{g})^{2}\right)_{n}
	\end{equation*}
	holds. See also Theorem \ref{thm:localindextheorem}.
\end{Thm}

The corresponding local index theorem for the Rarita--Schwinger operator is then just a special case of the Main Theorem, see Corollary \ref{cor:RS}. Similarly, we obtain the local index theorem for the higher Dirac and higher signature operators, see Corollaries \ref{cor:higherD} and \ref{cor:higherP}.

The proof follows the approach of Atiyah, Bott and Patodi \cite{atiyahHeatEquationIndex1973} to Gilkey's proof of the local index theorem:

For $t \searrow 0$, there is an asymptotic expansion of the heat kernel consisting of $n$-forms, i.e.
$$ \mathrm{str}(e^{-tD^2_g}(x,x)dvol_g)\sim \sum_{k=0}^\infty \omega_k(g)t^{\frac{t-n}{2}} \quad \text{for } t\searrow 0.$$
The theorem is proved by showing that for $k<n$ $\omega_k$ is zero and that for $k=n$ the form $\omega_n$ can be identified as the Chern--Weil form in question.

The first observation is that the construction of the geometric operators is in some sense a natural transformation from Riemannian metrics to elliptic first order differential operators. This implies that the coefficients $\omega_{k}$ of the asymptotic expansion of the heat kernel are natural transformations from Riemannian metrics to differential forms. If in normal coordinates, these differential forms can be expressed as a polynomial in the covariant derivatives of the curvature tensor (we call this property regularity) and are homogeneous under constant conformal rescalings $g\mapsto \lambda^{2}g$, $\lambda \in \mathbb{R}$, Gilkey's Theorem tells us that these natural transformations must be given by Pontryagin forms, see Section \ref{sec:invarianttheory}.

Thus the proof essentially consists of showing that the $\omega_k$ are have the required form in normal coordinates and are homogeneous under the aforementioned rescalings. In section \ref{sec:invarianttheory} we provide a brief overview of Gilkey's invariant theory, reformulated for our purposes. Section \ref{sec:geom_operators} introduces and discusses chiral geometric operators. We prove some properties that will be needed later in the proof of the local index theorem, and we consider some examples of chiral geometric operators.  In particular, we introduce higher signature operators (see Definition \ref{def:highersignatureoperator}) and compute their index. The regularity and homogeneity are dealt with in Section \ref{seq:hke}. The identification of the limit 
$$\lim_{t\searrow 0} \mathrm{str}(e^{-tD^2}(x,x)dvol_g)$$
is discussed in Section \ref{section:localindex}.

\section{Invariant theory}\label{sec:invarianttheory}

A category $\mathcal{C}$ is locally small, if for any two objects $c,c^{\prime}\in Obj(\mathcal{C})$, the morphisms $Mor(c,c^{\prime})$ form a set. We denote by $\mathbf{CAT}$ the category of locally small categories. The morphisms between two small categories $\mathcal{B}$, $\mathcal{C}$ are the functors $\mathcal{F}:\mathcal{B}\to \mathcal{C}$. A category is said to be discrete if its only morphisms are identities. A set $A$ defines naturally a discrete category (which we will again denote by $A$) whose objects are the elements of $A$. 

Let $\mathbf{Man}_n$ to be the category of compact, connected, smooth $n$-manifolds whose morphisms are local diffeomorphisms. Consider the functors
\begin{align*}
	\mathrm{Met}:\mathbf{Man}_n^{op}&\rightarrow \mathbf{CAT}\\
	\Omega^q:\mathbf{Man}_n^{op}&\rightarrow\mathbf{CAT},
\end{align*}
where $\mathrm{Met}$ sends a manifold $M$ to the set $\mathrm{Met}(M)$ of metrics on $M$ and $\Omega^{q}$ sends a manifold $M$ to the set $\Omega^{q}(M)$ of differential $q$-forms on $M$, both understood as locally small categories.

By an hermitian vector bundle we mean a triple $\xi=(E,h,\nabla)$, where $E\to M$ is a complex vector bundle, $h$ is a hermitian metric on $E$ and $\nabla$ is an affine connection compatible with $h$. Let 
$$
\mathrm{Vect}_{\mathbb{C}}^{k}: \mathbf{Man}_{n}^{op}\to \mathbf{CAT}
$$
be the contravariant functor mapping a manifold $M$ onto the discrete category $\mathrm{Vect}_{\mathbb{C}}^{k}(M)$ whose objects are the $k$-dimensional Hermitian vector bundles over $M$.

Let $\xi=(E,h,\nabla)$ be a hermitian vector bundle over an $n$-dimensional Riemannian manifold $(M,g)$, let $p \in M$ be an arbitrary point. An orthonormal basis $b_{1},\dots,b_{n}$ of $T_{p}M$ defines a normal coordinate system around $p$ and an orthonormal basis $f_{1},\dots,f_{k}$ of $E_{p}$ defines a synchronous frame by parallelly transporting along radial geodesics emanating from $p$, giving a local trivialization of $E$. We call such a pair of normal coordinates and synchronous frame a \textbf{normal trivialization centered at $p$}. Observe that normal trivializations are fixed up to an $\mathrm{O}(n)\times \mathrm{U}(k)$-action. With respect to such a trivialization, let $R^{i}_{jkl}$ be the components of the Riemann curvature tensor, $K^{\mu}_{\nu mn}$ be the components of the curvature of $E$ and $\Gamma^{\mu}_{i\nu}$ be the Christoffel symbols of $E$. For $\mu>0$, we define $\mu^{2}\xi$ to be the hermitian bundle $(E,\mu^{2}h,\nabla)$. 

\begin{defin}[\cite{atiyahHeatEquationIndex1973, atiyahErrataPaperHeat1975,freedAtiyahSingerIndex2021}]\label{def:invariant}
	A joint geometric invariant is a natural transformation 
  $$\omega:\mathrm{Met}\times \mathrm{Vect}_{\mathbb{C}}^{k}\rightarrow \Omega^q,$$
  such that in normal trivializations centered at a point $p$, the components of $\omega(g,\xi)(p)$ with respect to the basis $dx^{I}|_{p}=dx^{i_{1}}|_{p}\wedge\dots \wedge dx^{i_{q}}|_{p}$, $i_{1}<\dots<i_{q}$ are given by universal polynomial expressions of the covariant derivatives of the curvature tensors. We say that $\omega$ is homogeneous of weight $(k,l)$ if
  $$
  \omega(\lambda^2g,\mu^2 \xi)=\lambda^k\mu^l\cdot\omega(g,\xi),
  $$
\end{defin}

\begin{rmk}\label{rmk:regularity}
	The universality of the polynomial expressions means that the polynomial expressions are independent of the manifold and the hermitian vector bundle. The above condition can be reformulated as follows:
	We assume the existence of a universal polynomial 
	$$
	\hat{P}_{\omega}:\bigoplus_{j=0}^k(\mathbb{R}^n)^{\otimes j+4} \oplus \bigoplus_{j=0}^l\mathbb{C}^m \otimes \mathbb{C}^{m,*}\otimes(\mathbb{R}^n)^{\otimes j+2} \to \mathbb{C}\otimes\Lambda^q\mathbb{R}^n,$$ 
	such that in normal trivializations centered at a point $p$
	\begin{equation}\label{eq:regularity}
	  \omega(g,\xi)(p)=\hat{P}_{\omega}(R_{p},R_{p}^\prime,\dots,R_{p}^{(k)},K_{p},K_{p}^\prime\dots,K_{p}^{(l)}),
	\end{equation}
	where $R_{p}^{(i)}$ denotes the value of the $i$th covariant derivative of the Riemann curvature tensor at $p$ and similarly, $K^{(j)}_{p}$ denotes the value of the $j$th covariant derivative of the curvature tensor of the connection $\nabla^{\xi}$ at $p$. For an element $A$ of $O(n)\times U(m)$ we have
  \begin{align*}
  A^*\hat{P}_{\omega}(R_{p},R_{p}^\prime,\dots,R_{p}^{(k)},&K_{p},K_{p}^\prime\dots,K_{p}^{(l)}) = A^*\omega(g,\xi) \\
  &=\omega(A^*g,A^*\xi)\\
  &=\hat{P}_{\omega}(A^*R_{p},A^*R_{p}^\prime,\dots,A^{*}R_{p}^{(k)},A^*K_{p},A^*K_{p}^\prime\dots,A^*K_{p}^{(l)}).
  \end{align*}
  By averaging $\hat{P}_{\omega}$ over $O(n)\times U(m)$ we obtain a $O(n)\times U(m)$-equivariant polynomial $P_{\omega}$ which still fulfills Property \ref{eq:regularity}.
  \end{rmk}

\begin{rmk}
	The original definition given by Atiyah, Bott and Patodi in \cite{atiyahHeatEquationIndex1973} demands that in local coordinates, geometric invariants are given by polynomials in $(\det g)^{-1}$, $g_{ij}$, $\partial_\alpha g_{ij}$. As pointed out by Atiyah, Bott and Patodi in \cite{atiyahErrataPaperHeat1975}, this definition is too restrictive for our purposes. In Definition \ref{def:invariant}, we have weakened the original definition by requiring that the $q$-form can be expressed directly in terms of covariant derivatives of the curvature tensors. Also, we demand this property to hold only at the point $p$ at which the normal coordinates are centered. This shortens the overall proof of the local index theorem.
\end{rmk}

Let $G$ be a compact Lie group with Lie algebra $\mathfrak{g}$. We define the ring of (universal) characteristic classes $I(G)$ of $G$ to be the ring of polynomials on the Lie algebra $\mathfrak{g}$ of $G$ which are invariant under the adjoint representation,
\begin{equation*}
    I(G):= \{ q\text{ polynomial on }\mathfrak{g} \mid \forall g\in G:\, q(\mathrm{Ad}_g X)= q(X)  \}.
\end{equation*}
Let $q \in I(\mathrm{O}(n))$ be a universal characteristic class for the orthogonal group. For a Riemannian metric $g$ on a manifold $M$, the characteristic form $q(\nabla^{g}):= q(\Omega^{g})$, where $\Omega^{g}$ is the curvature $2$-form of $g$, is a globally defined differential form of $M$. Clearly, in local coordinates the components of the characteristic form $q(\nabla^{g})$ are given by polynomials in the Riemann curvature tensor. Since the characteristic forms behave nicely under the pullback of metrics, $q$ defines a geometric invariant
\begin{equation*}
  \omega_{q}: \mathrm{Met}\to \Omega^{*}, \quad g\mapsto q(\nabla^{g}).
\end{equation*}
Similarly, for $p \in I(\mathrm{U}(k))$, the characteristic form associated to $p$ defines a geometric invariant $\omega_{p}:\xi \mapsto p(\nabla^{\xi})$. 
\begin{defin}
  The ring of geometric invariants generated by the elements of \( I(\mathrm{O}(n)) \otimes I(\mathrm{U}(k)) \), with the assignment
  \[
  q \cdot p \mapsto \omega_{q} \wedge \omega_{p}, \quad q \in I(\mathrm{O}(n)),\, p \in I(\mathrm{U}(k)),
  \]
  is denoted by \( \mathrm{Pont} \otimes \mathrm{Chern} \). For a Riemannian metric $(M,g)$ and a hermitian vector bundle $\xi\to M$, we denote the image of $\mathrm{Pont} \otimes \mathrm{Chern}$ under evaluation at $(g,\xi)$ by $\mathrm{Pont}(g) \otimes \mathrm{Chern}(\xi) $.
\end{defin}

We can now state Gilkey's Theorem, which gives a characterization of regular, homogeneous Riemannian invariants:

\begin{thm}[Gilkey, \cite{atiyahHeatEquationIndex1973,gilkeyCurvatureEigenvaluesLaplacian1973,gilkeyInvarianceTheoryHeat1994}]\label{thm:Gilkey}
	A joint geometric invariant $\omega: \mathrm{Met}\times \mathrm{Vect}^{m}_\mathbb{C}\to \Omega^{*}$ of mixed weight $(k,l)$ vanishes identically if $k>0$ or $l\neq 0$ holds. If $k=l=0$ then $\omega$ is an element of $\mathrm{Pont}\otimes \mathrm{Chern}$. That is
	$$
	\omega=\begin{cases}
	0 & \text{if }k>0\text{ or }l\neq 0, \\
	\in \mathrm{Pont}\otimes \mathrm{Chern} & k=l=0.
	\end{cases}
	$$
\end{thm}

\section{Geometric operators}\label{sec:geom_operators}

\subsection{Geometric operators and their properties}

The term "geometric operator" is often used loosely to summarise in one word differential operators that appear naturally in geometry. For our purposes, it is useful to give a precise definition and to explore some basic properties. 

Denote by  $\mathbf{Man}_n$ the category of compact, connected, smooth $n$-manifolds whose morphisms consist of local diffeomorphisms.

We want to consider the following structure groups
$$G_n\in\{\mathrm{O}(n),\mathrm{SO}(n), \mathrm{Pin}(n), \mathrm{Spin}(n)\},$$
where $\mathrm{Pin}(n)$ is the $\mathrm{Pin}$-group in the Clifford-algebra $\mathrm{Cl}_n$ with the relation
$$
vw + wv =	-2 \langle v,w\rangle, \quad v,w \in \mathbb{R}^{n}.
$$
All of the above mentioned groups have in common that they have a natural unitary action on $\mathbb{R}^n$.

\begin{defin}[\cite{atiyahIndexEllipticOperators1968}]\label{def:groupstructure}
	Given a $n$-dimensional manifold $M$, a $G_n$-structure is a principal $G_n$-bundle $P$ over $M$, together with a vector bundle isomorphism
	$$
	\phi: P\times_{G_n} \mathbb{R}^n \xrightarrow{\cong} TM.
	$$
\end{defin}

\begin{rmk}
	Since the action of $G_{n}$ on $\mathbb{R}^{n}$ is orthogonal, a $G_n$-structure induces a Riemannian metric $g$ on the manifold by setting 
	\begin{equation*}
		g(\phi([p,v]),\phi([p,w]))= \left\langle v,w \right\rangle,
	\end{equation*}
	where $[p,v]\in P\times_{G_{n}}V$ denotes the equivalence class of $(p,v)\in P\times \mathbb{R}^{n}$ and $\left\langle \cdot,\cdot\right\rangle$ is the euclidean inner product.
	In the cases of $G_n$ being either $\mathrm{SO}(n)$ or $\mathrm{Spin}(n)$, we also obtain an induced orientation on $M$.
\end{rmk}

Given a local diffeomorphism $f:M\to N$ between two manifolds $M$ and $N$ and a $G_n$-structure $(P,\phi)$ on $N$, we can define a pullback of $(P,\phi)$ under the map $f$ as follows: First we take the pullback-bundle $f^*P$ of $P$ under $f$
$$
\begin{tikzcd}
	f^*P \arrow[r, "F"] \arrow[d] & P \arrow[d] \\
	M \arrow[r, "f"]              & N          
\end{tikzcd}
$$
to define a principal $G_n$-bundle over $M$. Since $f$ is a local diffeomorphism, its differential $df\colon TM \to TN$ defines a vector bundle map, whose fiber restrictions are invertible. The concatenation of maps
$$
f^*P \times_{G_n} \mathbb{R}^n \xrightarrow{F\times\mathrm{id}} P\times_{G_n}\mathbb{R}^n \xrightarrow{\phi}TN\xrightarrow{(df)^{-1}} TM
$$
yields the desired vector bundle isomorphism $f^*P\times_{G_n}\mathbb{R}^n\cong TM$.

\begin{defin}
	Letting $\mathbf{Man}_n^{G_n}$ be the subcategory  of $\mathbf{Man}_n$, whose objects are the manifolds admitting $G_n$-structures, we can define a functor
	$$\mathrm{Str}^{G_n}: (\mathbf{Man}_n^{G_n})^{\mathrm{op}}\rightarrow\mathbf{CAT},$$
	which sends a manifold $M$ to the discrete category $\mathrm{Str}^{G_{n}}(M)$ whose objects are the $G_n$-structures over $M$.
\end{defin}

The following table lists the subcategories $\mathbf{Man}_n^{G_n}$ for the different choices of $G_n$:
\begin{center}

\begin{tabular}[h]{l|l}
	$G_n$ & $Obj(\mathbf{Man}_n^{G_{n}})$ \\ 
	\hline
	$\mathrm{O}(n)$ &  $Obj(\mathrm{Man}_n)$ \\
	$\mathrm{SO}(n)$ & $M$ with $w_1(M)=0$ \\
	$\mathrm{Pin}(n)$ & $M$ with $w_1(M)^2+w_2(M)=0$ \\
	$\mathrm{Spin}(n)$ & $M$ with $w_1(M)=w_2(M)=0$ 
\end{tabular}
\end{center}
A $G_{n}$-structure $(P,\phi)$ on a manifold $M\in Obj(\mathbf{Man}_n^{G_{n}})$ induces a metric $g$ on $M$ and a covering map
$$
  \tilde{\phi} : P \to  \begin{cases}
  \mathrm{SO}(M,g) & \text{if } G_n = \mathrm{SO}(n), \mathrm{Spin}(n)\\
  \mathrm{O}(M,g) & \text{if } G_n = \mathrm{O}(n), \mathrm{Pin}(n),
  \end{cases}
$$
$\mathrm{SO}(M,g)$ being the oriented orthonormal frame bundle and $\mathrm{O}(M,g)$ being the orthonormal frame bundle of the Riemannian manifold $(M,g)$. The map $\tilde{\phi}$ is compatible with the $G_{n}$-action on the respective frame bundles.
The induced metric $g$ defines the Levi-Civita connection $1$-form $\omega^{LC}$ on $P$. For a $G_n$-representation $\rho: G_n\rightarrow \mathrm{End}(V)$, $\omega^{LC}$ induces a connection $\nabla^{LC}$ on the associated bundle $VM= P\times_\rho V$. A $G_n$-manifold is a tuple $(M,P)$, where $M\in Obj(\mathbf{Man}_n^{G_n})$, and $P\in \mathrm{Str}^{G_{n}}(M)$.

\begin{defin}\label{def:geomsymbol}
	A universal elliptic ($G_{n}$-)symbol is a $G_n$-equivariant map 
	$$\sigma:\mathbb{R}^n\rightarrow\mathrm{Hom}(V,W),$$
	where $V,W$ are hermitian representations of $G_n$, such that for all $\xi\in \mathbb{R}^n\setminus\{0\}$
	$$
	\sigma(\xi):V\to W
	$$
	is an isomorphism \cite{atiyahIndexEllipticOperators1968}. For any $G_n$-manifold $(M,P)$, the universal elliptic symbol $\sigma$ defines an elliptic first-order differential operator
	$$ D_\sigma:= \bar{\sigma}\circ \nabla^{LC}: C^\infty(M, VM)\rightarrow C^\infty(M,WM),$$
	where $\bar{\sigma}$ is the section of $T^\ast M \otimes \mathrm{Hom}(VM,WM)$ associated to $\sigma$.
	We call operators constructed in this way geometric first-order elliptic differential operators, or geometric operators for short.
\end{defin}

Next, we prove that the class of geometric operators is closed under taking adjoints.

\begin{lemma}\label{lem:adjointsofgeometricoperators}
	Given a universal elliptic symbol $\sigma: \mathbb{R}^{n}\to \mathrm{Hom}(V,W)$ the pointwise adjoint $\sigma^{*}$ is a again universal elliptic symbol. Given a $G_{n}$-manifold $(M,P)$ the adjoint of thee geometric operator $D_{\sigma}$ is given by $D_{-\sigma^{*}}$ and thus again a geometric operator.
\end{lemma}

\begin{proof}
	Let $\sigma:\mathbb{R}^{n}\to \mathrm{Hom}(V,W)$ be a universal elliptic symbol.
	Since $G_{n}$ acts by unitaries on $V$ and $W$, we have for $g \in G_{n}$
	\begin{align*}
		\sigma(g \cdot\xi)^{*}= (g\cdot \sigma(\xi)g^{-1})^{*} = g \sigma(\xi)^{*} g^{-1}.
	\end{align*}
	Thus, the adjoint is again a universal elliptic symbol. 

	Let $\phi \in C_{c}^{\infty}(M,VM)$, $\psi \in C_{c}^{\infty}(M,WM)$ be compactly supported sections. Define a vector field
	$X \in C_{c}^{\infty}(M;TM \otimes \mathbb{C})$ by
	\begin{equation*}
		\left\langle X,Y\right\rangle = \left\langle \bar{\sigma}(Y^{\flat})\phi,\psi\right\rangle,
	\end{equation*}
	where $\flat$ is the musical isomorphism. Let $e_{1},\dots,e_{n}$ be a local orthonormal frame. Then 
	\begin{align*}
		\operatorname{div} X &= \sum_{i=1}^{n} \left\langle \nabla_{e_{i}}X,e_{i}\right\rangle\\
		&= \sum_{i=1}^{n} \partial_{e_{i}} \left\langle X,e_{i}\right\rangle - \left\langle X, \nabla_{e_{i}}e_{i}\right\rangle\\
		&= \sum_{i=1}^{n} \partial_{e_{i}} \left\langle \bar{\sigma}(e_{i}^{\flat})\phi,\psi\right\rangle - \left\langle \bar{\sigma}((\nabla_{e_{i}}e_{i})^{\flat})\phi,\psi\right\rangle\\
		&= \sum_{i=1}^{n} \left\langle \nabla_{e_{i}}(\bar{\sigma}(e_{i}^{\flat})\phi),\psi\right\rangle + \left\langle \bar{\sigma}(e_{i}^{\flat})\phi, \nabla_{e_{i}} \psi\right\rangle - \left\langle \bar{\sigma}((\nabla_{e_{i}}e_{i})^{\flat})\phi,\psi\right\rangle\\
		&= \sum_{i=1}^{n}\left\langle \bar{\sigma}(e_{i}^{\flat})\nabla_{e_{i}}\phi,\psi\right\rangle + \left\langle \phi, \bar{\sigma}(e_{i}^{\flat})^{*}\nabla_{e_{i}}\psi\right\rangle\\
		&= \left\langle D_{\sigma}\phi,\psi\right\rangle - \left\langle \phi, D_{-\sigma ^{*}}\psi\right\rangle,
	\end{align*}
	where we have used that $\bar{\sigma}$ is a parallel section of $\mathrm{Hom}(T^{*}M\otimes VM,WM)$.
	The divergence theorem implies $D_{\sigma}^{*}= D_{-\sigma^{*}}$.
\end{proof}

We also want to consider twisted chiral geometric operators. Let $\xi=(E,h,\nabla^{\xi})$ be a hermitian vector bundle and $D$ a first order differential operator.
Then the twisted operator $D^{\xi}$ is constructed as follows:

Let $(b_{1},\dots,b_{n})$ be a local frame of $M$, then $D$ is given locally by
$$
D = \sum \bar{\sigma}(b^{i})\nabla^{g}_{b_{i}},
$$
where $\bar{\sigma}$ is the section of $\mathrm{Hom}(VM,WM)\otimes T^{*}M$ induced by $\sigma$ and $\nabla^{g}$ is the Levi-Civita connection on $VM$, $(b^{1},\dots,b^{n})$ is the dual frame to $(b_{1},\dots,b_{n})$. Then $D^{\xi}$ is defined by
$$
D^{\xi}= \sum (\bar{\sigma}(b^{i})\otimes \mathrm{id}_{E})\nabla_{b_{i}}^{g\otimes \xi}
$$
where $\nabla^{g\otimes \xi}$ is the tensor product connection on $VM\otimes E$, defined by
$$
\nabla^{g\otimes \xi}(s\otimes f)= \nabla^{g} s \otimes f + s \otimes \nabla^{\xi}f, \quad s \in C^{\infty}(M;VM), \, f\in C^{\infty}(M;E).
$$

For later purposes, we want to compute the total symbol of a twisted $G_n$-geometric operator in coordinates, induced from a universal elliptic symbol
\begin{equation*}
	\sigma : \mathbb{R}^n\rightarrow\mathrm{Hom}(V,W).
\end{equation*}
Let $(M,P)$ be a $G_{n}$-manifold, equipped with a hermitian vector bundle $\xi=(E,h,\nabla)$. Fix a point $p \in M$ and let $(x^1,...,x^n)$ be normal coordinates at $p$. In the case that $G_n\in\{\mathrm{SO}(n),\mathrm{Spin}(n)\}$ assume further that the coordinates are oriented. The coordinate vector fields $(\partial_1,...,\partial_n)$ define a local section of $\mathrm{GL}(M)$. Let $(e_{1},\dots,e_{n})$ be the synchronous frame obtained from $(\partial_{1},\dots,\partial_{n})$, that is, $e_{i}$ is obtained by parallel transport of $\partial_{i}|_{p}$ along radial geodesics. 
The base change $A:U \to \mathrm{GL}(n)$, defined by
\begin{equation}\label{eq:framechange}
	(e_1,...,e_n)(x)=(\partial_1,...,\partial_n)(x)A(x),
\end{equation}
is a smooth map, where $U$ is the image of the coordinate chart, with $A(0)=\mathrm{Id}$. For $G_n \in \{\mathrm{SO}(n) ,\mathrm{O}(n)\}$, $(e_1,...,e_n)$ defines a local section of $P$,
for $G_n \in \{\mathrm{Pin}(n), \mathrm{Spin}(n)\}$, we can lift $(e_1,...,e_n)$ to a local section $s$ of $P$. 
These define trivialisations of the bundles $VM$ and $WM$. With respect to these trivialisations and coordinates the geometric operator $D$ obtained from the universal elliptic symbol $\sigma$ has the form 
\begin{align*}
	D &= \sum_j \bar{\sigma}(e^j)\nabla_{e_{j}} = [s, \sum_j \sigma^j \left(\partial_{j} + \rho_\ast\left(\omega^{LC}\left(ds(e_j)\right)\right)\right)],
\end{align*}
where $\sigma^{j}= \sigma(\mathbf{e}_{j}^{*})$.
We have
\begin{equation*}
	\omega^{LC}\left(ds(e_j)\right)=\begin{cases}
		\frac{1}{4} \sum\limits_{\nu\neq\mu} \tilde{\Gamma}_{j\mu}^\nu \mathbf{e}_\mu\mathbf{e}_\nu\in\mathfrak{spin}(n) & \text{for }G_n\in \{\mathrm{Spin}(n),\mathrm{Pin}(n)\},\\
		\frac{1}{2}\sum\limits_{\nu\neq\mu}  \tilde{\Gamma}_{j\mu}^\nu\mathbf{E}_{\mu\nu} \in \mathfrak{so}(n) & \text{for }G_n\in \{\mathrm{SO}(n),\mathrm{O}(n)\},
	\end{cases}
\end{equation*}
where \(\tilde{\Gamma}_{j\mu}^\nu\) are the Christoffel symbols with respect to \((e_1,...,e_n)\) and 
\begin{equation*}
 	\mathbf{E}_{\mu\nu}= \begin{blockarray}{cccccc}
 		 & \substack{\mu\text{th}\\\text{col.}} &   & \substack{\nu\text{th}\\\text{col.}} & \\
 		\begin{block}{(ccccc)c}
 			 & \vdots&    &  \vdots&  &  \\ 
 			 \dots&  &\dots    & -1 & \dots &\substack{ \mu\text{th}\\ \text{row}} \\
 			 & \vdots &    &\vdots  &  \\
 			 \dots& 1 &\dots  &  & \dots& \substack{ \nu\text{th}\\ \text{row}}  \\
 			 & \vdots &  &\vdots  &  &  \\
 		\end{block}
 	\end{blockarray}
\end{equation*}
for $\mu <\nu$. In the case $\mu >\nu$, set $\mathbf{E}_{\mu\nu}= -\mathbf{E}_{\nu\mu}$. 
Using the Einstein summation convention, the Christoffel symbols are given by
 \begin{align*}
	 \tilde{\Gamma}_{j\mu}^\nu &= e^\nu(\nabla_{e_j}e_\mu)\\
	 &= (A^{-1})^\nu_\lambda dx^\lambda\left(A^k_j\nabla_{\partial_k}(A^l_\mu \partial_l)\right)\\
	 &=  A^k_j(A^{-1})^\nu_\lambda  dx^\lambda\left((\partial_k A^l_\mu )\partial_l +A^l_\mu\nabla_{\partial_k} \partial_l\right)\\
	 &= A^k_j  (A^{-1})^\nu_\lambda \left(\partial_k A^\lambda_\mu  +A^l_\mu\Gamma^\lambda_{kl}\right),
 \end{align*}
 where $\Gamma^\lambda_{kl}$ are the Christoffel symbols with respect to the coordinates $(x^1,\dots,x^n)$. Thus, in normal coordinates, $D$ is given by
 \begin{align*}
   D= A_{j}^i\sigma^j\partial_i+ A^k_j (A^{-1})^s_a \left(\partial_k A^{a}_r +A^l_r\Gamma^a_{kl}\right)\cdot \rho^{r}_{s}\sigma^j.
 \end{align*}

\begin{lemma}\label{lem:opincoord}
	Let $(M,P)$ be a $G_{n}$-manifold, equipped with a hermitian vector bundle $\xi=(E,h,\nabla)$ and let $\sigma$ be a universal elliptic symbol.
	In normal trivializations the twisted geometric operator associated to the symbol $\sigma$ is given by 
	\begin{align}\label{eq:opincord}
		D^{\xi}= A_{j}^i(\sigma^j\otimes \mathrm{id})\partial_i+  A^k_j (A^{-1})^s_a \left(\partial_k A^{a}_r +A^l_r\Gamma^a_{kl}\right)\cdot ((\rho^{r}_{s}\sigma^j)\otimes \mathrm{id}) + A^{i}_{j} (\Gamma^{\xi})^{\mu}_{i\nu} \cdot \sigma^{j}\otimes \mathcal{E}^{\nu}_{\mu},
	\end{align}
	where
	$$
		\rho^{\mu}_{\nu}= \begin{cases} 0 & \mu=\nu \\
		\frac{1}{4}\rho_\ast(\mathbf{e}_\mu\mathbf{e}_\nu) &\mu \neq \nu, \,G_n\in \{\mathrm{Spin}(n),\mathrm{Pin}(n)\},\\
		\frac{1}{2}\rho_\ast(\mathbf{E}_{\mu \nu}) &\mu \neq \nu, \,G_n\in \{\mathrm{SO}(n),\mathrm{O}(n)\},
	\end{cases}
	$$
	and $\Gamma^{\xi}$ are the Christoffel symbols of $\nabla^{\xi}$ with respect to the chosen normal trivialisation, $\mathcal{E}^{\nu}_{\mu} \in \mathrm{Mat}(m,\mathbb{C})$ is the matrix with entry $1$ in the $\nu$th column and $\mu$th row and zero elsewhere and $\mathrm{id}\in \mathrm{Mat}(m,\mathbb{C})$ is the identity matrix.
\end{lemma}

\begin{proof}
	The untwisted case has been discussed in the paragraph above. For the twisted case, let
	$$
	  D = \sum_{j} \sigma^{j}(\partial_{j}+ \Gamma_{j})
	$$
	be an operator on $\mathbb{R}^{n}$ acting on sections of a trivial complex bundle $\mathbb{R}^{n}\times\mathbb{C}^{m}\to \mathbb{R}^{n}$. Here, the $\Gamma_{j}$ are $m \times m$-matrices, whose entries correspond to Christoffel symbols of the connection induced bz the $\Gamma_{j}$. Let furthermore $\nabla^{\xi}$ be a connection with Christoffel symbols $((\Gamma^{\xi})^{\mu}_{i\nu})$ on a second vector bundle $\mathbb{R}^{n}\times \mathbb{C}^{k}\to \mathbb{R}^{n}$. The Christoffel symbols of the twisted connection $\tilde{\nabla}$ is then given by
	$$
	  \tilde{\Gamma}_{j} = \Gamma_{j}\otimes \mathrm{id}_{\mathbb{C}^{k}} + \mathrm{id}_{\mathbb{C}^{m}}\otimes (\Gamma^{\xi})^{\mu}_{i\nu} \cdot \mathcal{E}^{\nu}_{\mu}.
	$$
	Since the twisted operator is given by
	\begin{align*}
		D^{\xi} = \sum_{j} (\sigma^{j}\otimes \mathrm{id}_{\mathbb{C}^{k}}) (\partial_{j} + \tilde{\Gamma}_{j}),
	\end{align*}
	a short calculation shows Equation \ref{eq:opincord}.
\end{proof}
Since $(e_{1},\dots,e_{n})$ is orthonormal, the metric tensor $g_{ij}=g(\partial_{i},\partial_{j})$ with respect to the normal coordinates is given by
\begin{gather*}
  g_{ij} = \sum_{k} (A^{-1})^{k}_{i}(A^{-1})^{k}_{j}, \quad g^{ij}= \sum_{k} A^{i}_{k}A^{j}_{k}.
\end{gather*}
In Lemma \ref{lem:opincoord}, we have seen that in normal coordinates, the coefficients of a geometric operator are given by polynomials in the Christoffel symbols $\Gamma$ of the Levi-Civita connection, in the coefficients of $A$, $A^{-1}$ and their derivatives. Atiyah, Bott and Patodi \cite{atiyahHeatEquationIndex1973} gave explicit formulas for the Taylor expansions of $A$ and the Christoffel symbols $\Gamma^{\xi}$ of a hermitian vector bundle in normal trivializations, which we will recall now. Assume therefore, that $M$ is also equipped with a hermitian vector bundle $\xi$ and fix a normal trivialization.

For a function $f: B_{r}(0)\to \mathbb{C}$ defined on a ball of radius $r>0$ in $\mathbb{R}^{n}$ let
$$
  T_{0}(f) = \sum_{\alpha} \frac{1}{\alpha!} \frac{ \partial^{\alpha} f }{\partial x^{\alpha} }(0) x^{\alpha} \in \mathbb{C }[[x]]
$$
be the Taylor expansion at $0$. The homogeneous part of degree $k$ of $T_{0}(f)$ is denoted by $T_{0}(f)[k]$. 

\begin{prop}[\cite{atiyahHeatEquationIndex1973}]\label{prop:taylor1}
  The Taylor expansion of $A^{-1}$ is given by
  \begin{align*}
    T_{0}((A^{-1})^{i}_{j})[0] &= \delta_{ij}\\
    T_{0}((A^{-1})^{i}_{j})[1] &= 0 \\
    (n^2+n)T_{0}((A^{-1})^{i}_{j})[n]&= -2\sum_{k,l}x^kx^lT_{0}(R^{i}_{klj})[n-2], \quad n\geq 2.
  \end{align*}
  In particular, the derivatives of $A$, $A^{-1}$ and the metric tensor $g_{ij}$, $g^{ij}$ at $0$ can be expressed as polynomials of the derivatives of the Riemann curvature tensor $R^{i}_{jkl}$.
\end{prop}

\begin{rmk}
  The statements for $g_{ij}$, $g^{ij}$ follow since
  \begin{align*}
    g_{ij}= \sum_{k}(A^{-1})^{k}_{i}(A^{-1})^{k}_j,\quad g^{ij}= \sum_{k}A^{i}_kA^{j}_{k}.
  \end{align*}
  Since $R^{i}_{jkl} = g^{ij}R_{ijkl}$, we see that the derivatives of $A$, $A^{-1}$, $g_{ij}$, $g^{ij}$ at $0$ can be expressed by polynomials in the derivatives of $R_{ijkl}$. 
\end{rmk}

We have a similar Taylor expansion for the Christoffel symbols:

\begin{prop}[\cite{atiyahHeatEquationIndex1973}]\label{prop:taylor2}
	Let $\xi$ be a Hermitian vector bundle.
  	The Taylor expansion of the Christoffel symbols $\Gamma^{\xi}$ of $\xi$ with respect to any trivialization is given by
  	\begin{align*}
    	T_{0}((\Gamma^{\xi})^{\mu}_{k\nu})[0] &= 0\\
    	(n+1)T_{0}((\Gamma^{\xi})^{\mu}_{k\nu})[n] &= 2 \sum_{l}x^{l}T_{0}(K^{\mu}_{\nu lk})[n-1], \quad n\geq 1,
  	\end{align*}
	where $K$ is the curvature tensor of $\xi$.
  	In other words, the derivatives of $\Gamma^{\xi}$ at $0$ are given by linear expressions in the derivatives of $K^{\mu}_{\nu kl}$. In particular $\Gamma^{\xi}(0)=0$.
\end{prop}

The derivatives of the curvature tensors $R$, $K$ can be expressed as polynomials in the covariant derivatives of the curvature, the Christoffel symbols $\Gamma$, $\Gamma^{\xi}$ and their derivatives. 
We can apply Proposition \ref{prop:taylor2} to both $\xi$ and the tangent bundle, to recursively deduce that the derivatives of the curvature tensors at $0$ are polynomials in the covariant derivatives of the curvature tensors.

\begin{cor}\label{cor:taylor}
  The derivatives of $A$, $A^{-1}$, $g_{ij}$, $g^{ij}$, $\Gamma$ at $0$ can be expressed as polynomials of the covariant derivatives of the Riemann curvature tensor $R_{ijkl}$. The derivatives of $\Gamma^{\xi}$ at $0$ can be expressed as polynomials in the covariant derivatives of the curvature tensor $R_{ijkl}$, $K^{\mu}_{\nu kl}$ of $g$, $\xi$ respectively. 
\end{cor}

\begin{rmk}
  The polynomial expressions themselves are universal in the sense that they do not depend on the chosen normal trivializations.
\end{rmk}

\subsection{Generalized gradients and examples of geometric operators}

A similar concept to geometric operators is the one of generalized gradients introduced by Stein and Weiss \cite{steinGeneralizationCauchyRiemannEquations1968}. 

The finite-dimensional complex irreducible representations of a compact Lie group are parametrized by their dominant weights. In the case of $\mathrm{Spin}(n)$ these are given by tuples $\lambda\in \mathbb{Z}^m \cup (\frac{1}{2}+\mathbb{Z})^m$ whose coordinates fulfill the following condition:
\begin{equation*}
	\begin{cases}
		\lambda_1\geq \lambda_2\geq ...\geq \lambda_{m-1}\geq |\lambda_m|&\text{, if }n=2m,\\
		\lambda_1\geq \lambda_2\geq ...\geq \lambda_{m}\geq 0 & \text{, if }n=2m+1.
	\end{cases}
\end{equation*}
The irreducible representations of $\mathrm{SO}(n)$ correspond to those representations of $\mathrm{Spin}(n)$ which fulfill $\lambda \in \mathbb{Z}^m$. The irreducible representation with dominant weight $\lambda$ will be denoted by $V_\lambda$. The (complexification of the) $\mathrm{SO}(n)$-representation $\mathbb{R}^n$ has dominant weight $\tau=(1,0,...,0)$. Abusing notation, we will write $\mathbb{R}^{n}$ instead of $V_{\tau}$.

Let $\{ \epsilon_{1},\dots,\epsilon_{m} \}$ be the standard basis of $\mathbb{Z}^{m}$.
Consider an irreducible representation $V_\lambda$ of $G_n \in\{\mathrm{SO}(n), \mathrm{Spin}(n)\}$ with dominant weight $\lambda$. Fegan's selection rule \cite[Theorem 3.4]{feganConformallyInvariantFirst1976} implies that every summand in the decomposition of $\mathbb{R}^{n}\otimes V_\lambda$ occurs only with multiplicity one. Thus, the orthogonal projection
$$
\pi_{\lambda\mu}\colon \mathbb{R}^{n}\otimes V_\lambda \to V_\mu.
$$
is well-defined.
Going over to associated vector bundles of Riemannian manifolds with given geometric structures, the projections define a certain class of geometric operators:

\begin{defin}[\cite{bransonSteinWeissOperators1997, pilcaNewProofBransons2011}]
	Let $G_n \in\{\mathrm{SO}(n), \mathrm{Spin}(n)\}$,  $V_\lambda$ be an irreducible $G_{n}$-representation and let $(M,P)$ be a $G_n$-manifold. The covariant derivative $\nabla$ takes sections of $VM$ to sections of $T^*M \otimes VM$. 
	For each $V_\mu$ appearing in the decomposition of $\mathbb{R}^{n}\otimes V_\lambda$, there is a \textbf{generalized gradient} $D_{\lambda,\mu}$ defined by the composition
	$$
	\Gamma(V_\lambda M)\xrightarrow{\nabla}\Gamma(T^*M\otimes V_\lambda M)\xrightarrow{\pi_{\lambda\mu}}\Gamma(V_{\mu}M),
	$$	
	where $\pi_{\lambda\mu}$ is the orthogonal projection $\pi_{\lambda\mu}\colon T^*M\otimes V_\lambda M\to V_\mu M$.
\end{defin}

\begin{rmk}
	The difference between geometric operators and generalized gradients is that for a geometric operator we have not assumed that the representations on which the symbol is acting are irreducible and it was not asked to be a projection. Conversely, not every generalized gradient is elliptic. However, it follows easily from Schur's Lemma that every geometric operator is a linear combination of generalized gradients. 
\end{rmk}

We want to know in which cases generalized gradients or sums of generalized gradients define geometric operators, i.e. for which subsets
$$
I\subseteq \{\epsilon \mid V_{\lambda + \epsilon}\text{ occurs in }\mathbb{R}^{n}\otimes V_\lambda\}
$$
the operators
\begin{equation*}
	D_{\lambda,I} = \sum_{\epsilon\in I} D_{\lambda, \lambda+\epsilon},
\end{equation*}
are elliptic. Branson \cite{bransonSteinWeissOperators1997} studied a similar problem, investigating for which sets $I$ the operator $G_{\lambda,I}= D_{\lambda,I}^*D_{\lambda,I}$ is elliptic. This study led to the following classification result, which was also proven by Pilca \cite{pilcaNewProofBransons2011} using different methods. The operator $G_{\lambda,I}$ is said to be minimal elliptic, if $I$ contains no proper subset $J$ such that $G_{\lambda,J}$ is elliptic.

\begin{thm}[\cite{bransonSteinWeissOperators1997, pilcaNewProofBransons2011}]\label{thm:bransonclassification}
	Let $(M,g)$ be an $n$-dimensional Riemannian (spin) manifold and $\lambda$ be the dominant weight of an irreducible $\mathrm{SO}(n)$- (or $\mathrm{Spin}(n)$-) representation. Then $G_{\lambda,I}$ is minimal elliptic iff $I$ is one of the following sets:
	\begin{enumerate}
		\item if $n=2m+1$:
		\begin{enumerate}
			\item $\{\epsilon_1\}$,
			\item $\{0\}$, if $\lambda \in \left(\frac{1}{2}+\mathbb{Z}\right)^m$,
			\item $\{-\epsilon_i, \epsilon_{i+1}\}$, for $i\in \{1,...,m-1\}$,
			\item $\{-\epsilon_m,0\}$, if $\lambda\in \mathbb{Z}^m$
		\end{enumerate}
		\item if $n=2m$:
		\begin{enumerate}
			\item $\{\epsilon_1\}$
			\item $\{-\epsilon_m\}$ if $\lambda_m>0$
			\item $\{\epsilon_m\}$ if $\lambda_m<0$
			\item $\{- \epsilon_i, \epsilon_{i+1}\}$ for $i\in \{1,...,m-2\}$
			\item $\{-\epsilon_{m-1},\epsilon_m\}$, if $\lambda_m\geq 0$
			\item $\{-\epsilon_{m-1},-\epsilon_m\}$, if $\lambda_m\leq 0$
		\end{enumerate}
	\end{enumerate}
\end{thm}

Observe that the ellipticity of $G_{\lambda,I}$ only depends on $I$. Furthermore, $G_{\lambda,I}$ is elliptic
if and only if $D_{\lambda,I}$ is overdetermined elliptic, i.e. if for each $\xi\in T^*M\setminus \{0\}$ the principal symbol
$$
\sigma_{D_{\lambda,I}}(\xi)= \sum_{\epsilon\in I} \pi_{\lambda, \lambda+\epsilon}(\xi\otimes\cdot)
$$
is injective.

Suppose a generalized gradient $D_{\lambda,I}:\Gamma(V_{\lambda }M)\to \bigoplus_{\epsilon \in I}\Gamma(V_{\lambda+\epsilon}M)$ is elliptic. Let $
  \pi_{\lambda,I} = \oplus_{\epsilon \in I}\pi_{\lambda,\lambda+\epsilon}
$ be the corresponding projection. Then the adjoint
\begin{align*}
	\pi_{\lambda,I}^{*}: \mathbb{R}^{n}\otimes \bigoplus_{\epsilon \in I}V_{\lambda+\epsilon}&\to V_{\lambda}\\
	\xi \otimes w &\mapsto \pi_{\lambda,I}(\xi)^{*}w
\end{align*}
is $\mathrm{Spin}(n)$-equivariant and invertible for fixed $\xi$, and thus defines a geometric operator
$$
\tilde{D}_{\lambda,I}: \bigoplus_{\epsilon \in I}\Gamma(V_{\lambda+\epsilon}M)\to \Gamma(V_{\lambda}M).
$$
By Schur's Lemma, the restriction of $\pi_{\lambda,I}^{*}$ to $\mathbb{R}^{n}\otimes V_{\lambda+\epsilon}$ is a multiple of the projection
$$
\pi_{\lambda+\epsilon,\lambda}: \mathbb{R}^{n}\otimes V_{\lambda+\epsilon}\to V_{\lambda},
$$
and thus, the restriction of
$\tilde{D}_{\lambda,I}$ to $\Gamma(V_{\lambda+\epsilon}M)$ is the multiple of an overdetermined elliptic generalized gradient.
Therefore, the ellipticity of $D_{\lambda,I}$ implies that for each $\epsilon \in I$, the operator $D_{\lambda+\epsilon,\lambda}$ is an overdetermined elliptic generalized gradient and $\{ -\epsilon \}$ must be in the list of Theorem \ref{thm:bransonclassification}.

For $\lvert I \rvert = 2$, this leaves us only with the case $\dim M=4$ and
\begin{itemize}
	\item $\lambda_{1}-1 \geq\lambda_{2}\geq 0$, $I= \{ -\epsilon_{1}, \epsilon_{2} \}$,
	\item $\lambda_{1}-1 \geq-\lambda_{2}\geq 0$, $I= \{ -\epsilon_{1}, -\epsilon_{2} \}$.
\end{itemize}
Ellipticity of $D_{\lambda,I}$ would imply
\begin{align*}
 \dim V_{\lambda} = \dim V_{\lambda-\epsilon_{1}}+ \dim V_{\lambda+\epsilon_{2}} &\quad \text{for}\, \lambda_{2}\geq 0,\\
 \dim V_{\lambda} = \dim V_{\lambda-\epsilon_{1}}+ \dim V_{\lambda-\epsilon_{2}} &\quad \text{for}\, \lambda_{2}\leq 0.
\end{align*}
Invoking Weyl's dimension formula \cite[Theoem VI.1.7]{brockerRepresentationsCompactLie1985}, applied to $\mathrm{Spin}(4)$,
$$
  \dim V_{\mu} = (1+ \mu_{1}+\mu_{2})(1+\mu_{1}-\mu_{2}),
$$ 
we estimate for $\lambda_{1}-1 \geq\lambda_{2}\geq 0$
\begin{align*}
 \dim V_{\lambda-\epsilon_{1}} + \dim V_{\lambda+\epsilon_{2}}-\dim V_{\lambda} =& (2+\lambda_{1}+\lambda_{2})(\lambda_{1}-\lambda_{2})+ (\lambda_{1}+\lambda_{2})(\lambda_{1}-\lambda_{2})\\
 &- (1+\lambda_{1}+\lambda_{2})(1+\lambda_{1}-\lambda_{2})\\
 =& -2\lambda_{2}-1+\lambda_{1}^{2}-\lambda_{2}^{2}\\
 \geq & -\lambda_{1} -\lambda_{2} +\lambda_{1}^{2}-\lambda_{2}^{2} \\
 =& \left(\lambda_{1}- \frac{1}{2}\right)^{2} - \left(\lambda_{2}+\frac{1}{2}\right)^{2} \geq 0
\end{align*}
with equality if and only if $\lambda_{1}= \lambda_{2}+1$. For $\lambda_{1}-1 \geq-\lambda_{2}\geq 0$ we obtain similarly
\begin{align*}
	\dim V_{\lambda-\epsilon_{1}} + \dim V_{\lambda-\epsilon_{2}}-\dim V_{\lambda} &\geq \left(\lambda_{1}-\frac{1}{2}\right)^{2} -\left(\lambda_{2}-\frac{1}{2}\right)^{2} \geq 0
\end{align*}
with equality if and only if $\lambda_{1}= 1-\lambda_{2}$.

For the case $\lvert I \rvert=1$, the only candidates for elliptic generalized gradients are
\begin{itemize}
	\item $D_{\lambda,\{ 0 \}}$ for $\dim M= 2m+1$, $\lambda \in\left( \frac{1}{2}+\mathbb{Z} \right)^{m}$,
	\item $D_{\lambda,\{ -\epsilon_{m} \}}$ for $\dim M=2m$, $\lambda_m>0$,
	\item $D_{\lambda,\{ \epsilon_{m} \}}$ for $\dim M=2m$, $\lambda_m<0$.
\end{itemize}
It is clear that in the odd-dimensional case $D_{\lambda,\{ 0 \}}$ is indeed elliptic. 
In the even-dimensional case, the ellipticity for $D_{\lambda,\{ -\epsilon_{m} \}}$ with $\lambda_{m}> 0$ implies the ellipticity of $D_{\lambda-\epsilon_{m},\{ \epsilon_{m} \}}$. The latter operator is only overdetermined elliptic if $\lambda_{m}-\epsilon_{m}<0$.
Consequently, $D_{\lambda,{-\epsilon_{m}}}$ is elliptic if and only if $\lambda$ is such that $\lambda_{m}=\frac{1}{2}$.  By analogy, we get that $D_{\lambda,\epsilon_{m}}$ is elliptic for $\lambda=-\frac{1}{2}$.

\begin{cor}\label{cor:ellipticgeneralizedgradients}
	The only elliptic generalized gradients are:
	\begin{itemize}
		\item $\dim M=4$, $D_{(\mu+1,\mu), \{ -\epsilon_{1},\epsilon_{2} \}}$ for $\mu\geq 0$,
		\item $\dim M=4$, $D_{(\mu+1, -\mu), \{ -\epsilon_{1},-\epsilon_{2} \}}$ for $\mu\geq 0$,
		\item $\dim M= 2m+1$, $D_{\lambda,\{ 0 \}}$ for $\lambda \in(\frac{1}{2}+\mathbb{Z})^{m}$,
		\item $\dim M=2m$, $D_{\lambda,{-\epsilon_{m}}}$ for $\lambda_{m}= \frac{1}{2}$,
		\item $\dim M=2m$, $D_{\lambda,{\epsilon_{m}}}$ for $\lambda_{m}= -\frac{1}{2}$.
	\end{itemize}
\end{cor}

\begin{defin}[\cite{buresEigenvaluesConformallyInvariant1999}]
	Let $\lambda$ be a dominant weight with $\lambda_{m}=\frac{1}{2}$. In odd dimensions, the higher Dirac operator $D_{\lambda}$ is defined to be the operator $D_{\lambda,\{ 0 \}}$. In even dimensions, the higher Dirac operator $D_{\lambda}$ is defined to be the operator
	\begin{equation*}
		D_{\lambda}= \begin{pmatrix}
		0 & D_{\lambda,-}\\
		D_{\lambda,+} & 0
		\end{pmatrix}
	\end{equation*}
	where $D_{\lambda,+}= D_{\lambda, -\epsilon_{m}}$ and $D_{\lambda,-}= D_{\bar{\lambda},\epsilon_{m}}$, where $\bar{\lambda}$ is the dominant weight
	\begin{equation*}
		\bar{\lambda}=(\lambda_{1},\dots,\lambda_{m-1},-\lambda_{m}).
	\end{equation*}
	The operator $D_{\lambda}$ is an elliptic first order differential operator acting on the vector bundle $V_{\lambda}\oplus V_{\bar{\lambda}}$.
\end{defin}

\begin{rmk}\label{rmk:higher_Dirac_operator}
	For $\lambda=(\frac{1}{2},\dots,\frac{1}{2})$ we obtain the classical Dirac operator, for $\lambda=(\frac{3}{2},\frac{1}{2},\dots,\frac{1}{2})$ we obtain the Rarita--Schwinger operator. A subclass of higher Dirac operators appearing in the study of spinor-valued differential forms \cite{sommenMonogenicDifferentialForms1992,delangheCliffordAlgebraSpinorValued1992} are the operators associated to the weights
	\begin{equation*}
		\lambda_{j} =\big(\underbrace{\frac{3}{2},\dots,\frac{3}{2}}_{j\text{ times}},\frac{1}{2},\dots,\frac{1}{2}\big), \quad j\in \{ 0,\dots,m-1 \}.
	\end{equation*}
	We denote the corresponding Dirac operators by $D_{j}$.
\end{rmk}

The operators $(D_{(\mu+1,\pm\mu), \{ -\epsilon_{1},\pm\epsilon_{2} \}})_{\mu\geq 0}$ form a special class of elliptic generalized gradients that only appears in dimension four. The simplest members of this class, $D_{(1,0),\{ -\epsilon_{1},\pm \epsilon_{2} \}}$, are given by
\begin{equation*}
	\Omega^{1}(M)\ni \omega\mapsto (\delta \omega,\frac{1}{2}(1\pm*)d\omega)\in \Omega^{0}(M)\oplus \Omega^{2}_{\pm}(M),
\end{equation*}
where $\Omega^{k}(M)$ denotes the space of differential forms, $\Omega^{2}_{\pm}(M)$ the space of (anti-) self-dual differential forms, $d$ is the exterior differential and $\delta$ is the codifferential. We can use the operators $(D_{(\mu+1,\pm\mu), \{ -\epsilon_{1},\pm\epsilon_{2} \}})_{\mu\geq 0}$ to define a sequence of operators that has the signature operator as its first term.
Consider to this end the operators
\begin{align*}
	\tilde{P}_{\mu}^{-}:= D_{(\mu+1,\mu), \{ -\epsilon_{1},\epsilon_{2} \}} &  \colon V_{(\mu+1,\mu)} \to V_{(\mu,\mu)}\oplus V_{(\mu+1, \mu+1)} \\
	\tilde{P}_{\mu}^{+}:= D_{(\mu+1,-\mu), \{ -\epsilon_{1},-\epsilon_{2} \}}  &  \colon V_{(\mu+1,-\mu)}  \to V_{(\mu,-\mu)}\oplus V_{(\mu+1, -(\mu+1))}
\end{align*}
for $\mu\geq 0$
and define the $\mathrm{SO}(4)$-representations
\begin{align*}
	W_{\mu}^{+}  & := V_{(\mu,\mu)} \oplus V_{(\mu+1,-\mu)}\oplus V_{(\mu+1,\mu+1)}, \\
	W_{\mu}^{-} & := V_{(\mu,-\mu)} \oplus V_{(\mu+1,\mu)} \oplus V_{(\mu+1,-(\mu+1))}, \\
	W_{\mu} & := W_{\mu}^{+}\oplus W_{\mu}^{-}. 
\end{align*}
\begin{defin}\label{def:highersignatureoperator}
	Let $(M,g)$ be an oriented four-dimensional Riemannian manifold.
	For $\mu\geq 0$, the higher signature operator $P_{\mu}$, acting on sections of the to $W_{\mu}$ associated vector bundle, is given by
	$$
	  P_{\mu}:= \begin{pmatrix}
		& \tilde{P}_{\mu}^{-}\oplus (\tilde{P}_{\mu}^{+})^{*} \\
		\tilde{P}_{\mu}^{+}\oplus (\tilde{P}_{\mu}^{-})^{*}
	  \end{pmatrix}
	$$
	with respect to the splitting $W_{\mu} = W_{\mu}^{+}\oplus W_{\mu}^{-}$.
\end{defin}

\begin{rmk}
	We have defined $P_{\mu}$ in such a way that for $\mu=0$ the representations $W_{0}^{\pm}$ are canonically isomorphic to the space of self-dual/anti-self-dual forms, and that under these identifications, the operator $P_{0}$ coincides with the full signature operator $d+ \delta$.
\end{rmk}

\subsection{Chiral geometric operators}\label{sec:chiral_geometric_operators}
In order to prove the local index theorem, we need to introduce the notion of chirality. Namely, given a universal elliptic symbol $\sigma : \mathbb{R}^{n}\otimes V \rightarrow V$ we need that the orientation of the manifold defines a $\mathbb{Z}_2$-grading on $V$, with respect to which the symbol is odd.

\begin{defin}\label{def:chiralgeomsymbol}
	Let $H_n\in\{ \mathrm{Pin}(n),\mathrm{O}(n)\}$, $V$ an $H_n$-representation and let $G_n\subseteq H_n$ be the connected component of the neutral element. 
	A chiral universal elliptic ($G_n$-)symbol is given by a pair $(\sigma,\varepsilon)$ of an universal elliptic $H_n$-symbol $\sigma: \mathbb{R}^n\to \mathrm{End}(V)$ and an $H_n$-equivariant map
	\begin{equation*}
		\varepsilon: \Lambda^n\mathbb{R}^n\rightarrow\mathrm{End}(V),
	\end{equation*}
	such that $\sigma(\xi)$ is skew-adjoint for all $\xi\in\mathbb{R}^n$ as well as
	\begin{align*}
		\varepsilon(\mathbf{e}_1\wedge...\wedge\mathbf{e}_n)^2= 1 \quad\text{and}\quad
		\sigma\cdot \varepsilon(\mathbf{e}_1\wedge\cdots \wedge \mathbf{e}_n)=-\varepsilon(\mathbf{e}_1\wedge\dots \wedge \mathbf{e}_n)\cdot\sigma,
	\end{align*}
	where $\mathbf{e}_1,...,\mathbf{e}_n$ is the standard basis of $\mathbb{R}^n$. A geometric operator obtained from a chiral geometric symbol is said to be a chiral geometric operator.
\end{defin}

\begin{example}
	The most classical examples of chiral geometric operators are the signature operator and the Dirac operator on an oriented $n=2m$-dimensional manifold. The signature operator arises from the elliptic $\mathrm{O}(n)$-symbol
	\begin{equation*}
		\sigma_{\mathrm{sign}}:\mathbb{R}^{n}\times \Lambda ^{*}\mathbb{C}^{n}\to \Lambda ^{*}\mathbb{C}^{n}, \quad \sigma(\xi)\omega:= \xi \wedge \omega- \iota_{\xi}\omega, 
	\end{equation*}
	where $\iota_{\xi}$ is the insertion operator and $\epsilon$ is given by 
	\begin{equation*}
		\epsilon_{\mathrm{sign}}(e_{1}\wedge\dots \wedge e_{n})(\psi)= i^{k(k-1)+m}\iota_{\psi}(e_{1}\wedge\dots \wedge e_{n}), \quad \psi \in \Lambda^{k}\mathbb{C}^{n}.
	\end{equation*} 
	The elliptic $\mathrm{Pin}(n)$-symbol of the Dirac operator is given by Clifford multiplication on the complex spinor space $\Sigma_n$ and $\varepsilon$ is given by 
	$$\varepsilon(\mathbf{e}_1\wedge\dots\wedge \mathbf{e}_n) = \omega_\mathbb{C} = i^m \mathbf{e}_1\cdot \dots \cdot \mathbf{e}_n\in \mathbb{C}l_n.$$
\end{example}

Let $G_n \in\{ \mathrm{SO}(n),\mathrm{Spin}(n)\}$ and let $(\sigma,\varepsilon)$ be a universal elliptic $G_n$-symbol.
Let $V^\pm$ be the $\pm 1$-eigenspaces of $\varepsilon(\mathbf{e}_1\wedge \dots \wedge \mathbf{e}_n)$. Then $\sigma(\mathbb{R}^{n}\otimes V^\pm)=V^\mp$ and $V^\pm$ are $G_n$-representations. The restrictions
$$\sigma^\pm : \mathbb{R}^n\otimes V^\pm\rightarrow V^\mp$$
are $G_n$-equivariant. Hence, given a Riemannian $G_n$-manifold $(M,P)$, the chiral universal elliptic symbol induces a $\mathbb{Z}_2$-grading $VM = V^+M\oplus V^-M$. Furthermore $V^\pm M$ can be identified as the $\pm 1$-eigenspaces of $\bar{\varepsilon}(\mathrm{dvol}_g)$. We obtain the geometric operators 
\begin{align*}
	D_{\sigma}&\colon C^\infty(M,VM)\rightarrow C^\infty(M,VM), \\ D_{\sigma^+}&\colon C^\infty(M,V^+M)\rightarrow C^\infty(M,V^-M), \\ D_{\sigma^-} &\colon C^\infty(M,V^-M)\rightarrow C^\infty(M,V^+M), 
\end{align*}
such that 
\begin{equation}\label{eq:splitting}
	D_\sigma = \begin{pmatrix}
		0 & D_{\sigma^-}\\ D_{\sigma^+} & 0 
	\end{pmatrix}.
\end{equation}
Since $D$ is formally self-adjoint, we have that $(D_{\sigma^+})^*= D_{\sigma^-}$.

Given an $H_n$-structure $P$ on an orientable, connected manifold $M$, where $H_n\in\{\mathrm{O}(n),\mathrm{Pin}(n)\}$, the choice of an orientation 
$$\smallO\in \pi_0(\mathrm{GL}(M))=\pi_0(\mathrm{O}_g(M))$$
induces a $G_n$-structure $P_{\smallO}$, where $G_n$ is the identity component of $H_n$. Namely, the connected component $P_{\smallO}$ of $P$ that covers the connected component $\smallO$ of $\mathrm{O}(M)$ is indeed a principal $G_n$-bundle which defines a $G_n$-structure on $M$.

In general, if $M$ is orientable, $\pi_0(\mathrm{GL}(M))$ consists of two elements. For a choice of orientation $\smallO$, the complementary element will be denoted by $\bar{\smallO}$.

\begin{prop}\label{prop:chirality}
	Let $H_n$, $G_n$ be as above and $(\sigma,\varepsilon)$ be a chiral universal elliptic symbol acting on an $H_n$-representation $V$. Let $(M,P)$ be an orientable $H_n$-manifold and $\smallO$ be an orientation of $M$. Denote the induced $G_n$-structures by $P_{\smallO}$, $P_{\bar{\smallO}}$.
	Then we have the following equalities:
	\begin{align*}
		V_{\smallO}M= V_{\bar{\smallO}}M=VM,\quad V^\pm_{\smallO}M=V^\mp_{\bar{\smallO}}M, \quad
		D_{\sigma^\pm,\smallO}=D_{\sigma^\mp,\bar{\smallO}}.
	\end{align*}
\end{prop}

\begin{proof}
	An element $h\in H_n\setminus G_n$ defines a diffeomorphism 
	\begin{align*}
		P_{\smallO}&\rightarrow P_{\bar{\smallO}}\\
		p&\mapsto p\cdot h
	\end{align*}
	Furthermore, we have canonical inclusions 
	$$ \iota_{\smallO} : V_{\smallO}M\rightarrow VM, \quad \iota_{\bar{\smallO}} : V_{\bar{\smallO}}M\rightarrow VM$$
	induced from the inclusions $P_{\smallO}, P_{\bar{\smallO}}\hookrightarrow P$.
	But every element of $[p,v]\in VM$ can be represented by an element of $ V_{\smallO}M$ since
	$$[p,v]= [p\cdot h, h^{-1}v]\in V_{\smallO}M \quad \text{if }p\notin P_{\smallO},$$
	the same holds for $ V_{\bar{\smallO}}M$.
	Consequently, $\iota_{\smallO}$ and $\iota_{\bar{\smallO}}$ are vector bundle isomorphisms and hence the first equality follows. 
	
	The other two equalities follow from the fact that $V^\pm_{\smallO}M$ is given by the $\pm 1$-eigenspaces of $\bar{\varepsilon}(\mathrm{dvol}_{g,\smallO})$ and the fact that 
	\begin{equation*}
		\mathrm{dvol}_{g,\smallO}= -\mathrm{dvol}_{g,\bar{\smallO}}.\qedhere
	\end{equation*}
\end{proof}

\subsection{Examples of chiral geometric operators}
For the construction of chiral geometric operators, we give a short recap of the representation theory for $\mathrm{O}(n)$ and $\mathrm{Pin}(n)$ following \cite[Chapter VI.7]{brockerRepresentationsCompactLie1985}.

Let $H_{n}\in \{ \mathrm{O}(n), \mathrm{Pin}(n) \}$ and $G_{n}\in \{ \mathrm{SO}(n), \mathrm{Spin}(n) \}$ such that $G_{n}\subseteq H_{n}$, $n=2m$. 
For a $G_{n}$-module $V$ and $h\in H_{n}$ we can define a twisted $G_{n}$-module $V_{h}$ by
$$
G_{n}\times V\to V, \quad (g,v)\mapsto hgh^{-1}v.
$$
The isomorphism type of $V_{h}$ depends only on the coset $hH$, consequently $(V_{x})_{y}=V_{xy}$ and $V_{1}=V$. Thus, $H_{n}/G_{n}\cong \mathbb{Z}_{2}$ acts on the set of dominant weights of $\mathrm{Spin}(n)$. 

Now let $U$ be an $H_{n}$-module. Let $x$ be the generator of $H_{n}/G_{n}$ and let $\Omega(k)$, $k\in \{ 0,1 \}$ be the representation 
$$
H_{n}/G_{n}\times \mathbb{C}\to \mathbb{C}, \quad (x,z)\mapsto \exp(\pi ikx)\cdot z,
$$
where we have identified $H_{n}/G_{n}$ with $\mathbb{Z}_{2}$.
Then, $\Omega(k)$ is also an $H_{n}$-representation and we obtain $H_{n}$-modules $U \cong U\otimes \Omega(0)$ and $U\otimes \Omega(1)$.
It turns out that if $V$ was an irreducible $G_{n}$-module, then so is $V_{x}$ and if $U$ is an irreducible $H_{n}$-module, then so is $U\otimes \Omega(k)$.
We can now divide the irreducible $H_{n}$- and $G_{n}$-modules into two types:

\begin{center}
	\begin{tabular}[h]{c|c|c}
		& type $\mathrm{I}$ & type $\mathrm{II}$ \\ 
		\hline
		 $G_{n}$-modules $V$ & all $V_{x}$ isomorphic & all $V_{x}$ distinct \\
		 $H_{n}$-modules $U$ & all $U\otimes \Omega(k)$ distinct  & $U\otimes \Omega(k)$ isomorphic 
	\end{tabular}
\end{center}

Let $\mathrm{res}_{G_{n}}U$ be the restriction of an $H_{n}$-module $U$ to $G_{n}$ and $\mathrm{ind}_{G_{n}}V$ be the $H_{n}$-module induced from $V$. This can be realised, for example, by
$$
\mathrm{ind}_{G_{n}}V = \{ f\colon H_{n} \to V\text{ cont.}\mid f(hg^{-1})= gf(h) \, \forall  g\in G_{n}, h\in H_{n}\}
$$
with $H_{n}$ acting by $(hf)(x)= f(h^{-1}x)$.

\begin{thm}[\cite{brockerRepresentationsCompactLie1985}]\label{thm:restriction/inductionofmodules}
	If $U$ is an irreducible $H_{n}$-module of type $\mathrm{I}$, then $\mathrm{res}_{G_{n}}U=V$ is an irreducible $G_{n}$-module of type $\mathrm{I}$ with $\mathrm{ind}_{G_{n}}V \cong \bigoplus_{k\in \mathbb{Z}_{2}}U\otimes \Omega(k)$.
	If $U$ in an irreducible $H_{n}$-module of type $\mathrm{II}$, then $\mathrm{res}_{G_{n}}U\cong \bigoplus_{x \in H_{n}/G_{n}}V_{x}$ with $V_{x}$ irreducible of type $\mathrm{II}$ and $\mathrm{ind}_{G_{n}}V_{x}\cong U$ for all $x \in H_{n}/G_{n}$.
\end{thm}

\begin{prop}[\cite{brockerRepresentationsCompactLie1985}]\label{prop:dominantweightO(n)}
	An irreducible $G_{n}$-module $V_{\lambda}$ with dominant weight $\lambda=(\lambda_{1},\dots,\lambda_{m})$ is of type $\mathrm{I}$ if and only if $\lambda_{m}=0$. If $V_{\lambda}$ is of type $\mathrm{II}$, and for $x \in H_{n}\setminus G_{n}$, the twisted $G_{n}$-module $(V_{\lambda})_{x}$ has dominant weight $\bar{\lambda}:=(\lambda_{1},\dots, \lambda_{m-1},-\lambda_{m})$.
\end{prop}

\subsubsection{Higher Dirac operators}

For a dominant weight $\lambda$ with $\lambda_{m}> 0$, let $U_{\lambda}= \mathrm{ind}_{G_{n}} V_{\lambda}$. Denoting $\bar{\lambda} = (\lambda_{1},\dots,\lambda_{m-1},-\lambda_{m})$,
the equation
\begin{equation*}
	\mathrm{res}_{G_{n}} U_{\lambda} = V_{\lambda}\oplus V_{\bar{\lambda}}
\end{equation*} 
holds by Proposition \ref{prop:dominantweightO(n)}. Let $\varepsilon: U_{\lambda}\to U_{\lambda}$ be the homomorphism defined by 
\begin{equation*}
	\varepsilon = \mathrm{id}_{V_{\lambda}} \oplus - \mathrm{id}_{V_{\bar{\lambda}}}.
\end{equation*}
By definition, the homomorphism $\varepsilon$ is $G_{n}$-invariant. An element $g\in H_{n}\setminus G_{n}$ maps $V_{\lambda}$ into $V_{\bar{\lambda}}$ and vice versa, and thus
\begin{equation*}
	g \cdot\varepsilon \cdot g^{-1} = \det(g) \cdot\varepsilon, \quad g\in H_{n}.
\end{equation*}
The map $\varepsilon$ therefore induces an $H_{n}$-equivariant map
\begin{equation*}
	\epsilon:\Lambda^{n}\mathbb{R}^{n}\to U_{\lambda}\otimes U_{\lambda}^{*}.
\end{equation*}

Let $\lambda$ be a dominant weight with $\lambda_{m}=\frac{1}{2}$ and $D_{\lambda}$ be the corresponding higher Dirac operator. It is an $\mathrm{Pin}(n)$-geometric operator, its symbol $\sigma_{\lambda}$ is given by the projection
\begin{equation*}
	U_{\lambda}\otimes \mathbb{R}^{n} \to U_{\lambda}.
\end{equation*}
Since $D_{\lambda}$ intertwines $V_{\lambda}$ and $V_{\bar{\lambda}}$, the operator is odd with respect to $\epsilon$. 

\begin{prop}\label{prop:higherdiracoperators}
	For a dominant weight $\lambda$ with $\lambda_{m}=\frac{1}{2}$, the higher Dirac operator $D_{\lambda}$ is a chiral geometric operator. 
\end{prop}

\begin{proof}
	The only thing that remains to be checked is the self-adjointness of $D_{\lambda}$. 
	The operator $D_{\lambda}$ can be alternatively obtained from the classical Dirac operator $D^{\lambda^{\prime}}$ twisted with the bundle induced from the $\mathrm{Spin}(n)$-module $V_{\lambda^{\prime}}$, where
	\begin{equation*}
		\lambda^{\prime}= \big(\lambda_{1}-\frac{1}{2},\dots,\lambda_{m-1}-\frac{1}{2},0\big),
	\end{equation*}
	see e.g. \cite{buresEigenvaluesConformallyInvariant1999}.
	The representation $V_{\lambda}$ appears with multiplicity one in the module 
	$$
	  \Sigma_{n}\otimes V_{\lambda^{\prime}},
	$$
	and $D_{\lambda}$ is obtained from $D^{\lambda^{\prime}}$ by the concatenation
	\begin{equation*}
		V_{\lambda}\otimes \mathbb{R}^{n}\hookrightarrow \Sigma_{n}\otimes V_{\lambda^{\prime}}\otimes \mathbb{R}^{n} \xrightarrow{\mathrm{cl}\otimes \mathrm{id_{V_{\lambda^{\prime}}}}} \Sigma_{n}\otimes V_{\lambda^{\prime}} \twoheadrightarrow V_{\lambda},
	\end{equation*}
	$\mathrm{cl}$ being the Clifford multiplication.
	The self-adjointness of the Dirac operator implies the self-adjointness of $D_{\lambda}$.
\end{proof}

\subsubsection{Higher signature operators}

For $\mu \in \mathbb{N}_{0}$, consider the higher signature operator $P_{\mu}$. It is clear from Proposition \ref{prop:dominantweightO(n)} that for $x \in \mathrm{O}(4) \setminus \mathrm{SO}(4)$ 
\begin{align*}
	(V_{(\mu,\pm \mu)})_{x} \cong V_{(\mu,\mp \mu)}, \quad (V_{(\mu+1,\pm \mu)})_{x} \cong V_{(\mu+1,\mp \mu)},\quad (V_{(\mu+1,\pm (\mu+1))})_{x} \cong V_{(\mu+1,\mp (\mu+1))}.
\end{align*}
Defining
\begin{align*}
	V_{\mu}^{1} = \mathrm{ind}_{\mathrm{SO}(4)} V_{(\mu,\mu)}, \quad V_{\mu}^{2} = \mathrm{ind}_{\mathrm{SO}(4)} V_{(\mu+1,\mu)}, \quad V_{\mu}^{3} = \mathrm{ind}_{\mathrm{SO}(4)} V_{(\mu+1,\mu+1)},
\end{align*}
Fegan's selection rule \cite[Theorem 3.4]{feganConformallyInvariantFirst1976} implies that $V_{\mu}^{1}$ and $V_{\mu}^{3}$ appear with multiplicity one in the $\mathrm{O}(4)$-representation $\mathbb{R}^{4}\otimes V_{\mu}^{2}$. Thus there exists an $\mathrm{O}(4)$-equivariant projection
$$
  \tilde{\sigma} : \mathbb{R}^{4}\otimes V_{\mu}^{2}\to V_{\mu}^{1}\oplus V_{\mu}^{3},
$$
which induces the generalized gradient $\tilde{P}_{\mu}$. As a $\mathrm{SO}(4)$-geometric operator, the operator splits into
$$
  \tilde{P}_{\mu} = \tilde{P}_{\mu}^{+}\oplus \tilde{P}_{\mu}^{-}.
$$
The higher signature operator is then given by $P_{\mu}= \tilde{P}_{\mu}\oplus \tilde{P}_{\mu}^{*}$, the discussion shows that it is a $\mathrm{O}(4)$-geometric operator acting on sections of the vector bundle $W_{\mu}= V^{1}_\mu \oplus V^{2}_\mu \oplus V^{3}_\mu$. It is self-adjoint by definition.
As for the higher Dirac operator, we can set the chirality operator $\epsilon$ to
\begin{equation*}
	\varepsilon = \mathrm{id}_{W_{\mu}^{+}} \oplus - \mathrm{id}_{W_{\mu}^{-}}.
\end{equation*}
It is easily checked that it satisfies all the requirements of Definition \ref{def:chiralgeomsymbol}, the name is justified by the fact that $P_{1}$ is the usual signature operator.

We can calculate the index of $P_{\mu}$ for $\mu\geq 1$ as follows: 
Steinberg's multiplicity formula and Fegan's selection rule imply the splittings
\begin{align*}
	V_{(\mu,\pm \mu)}\otimes V_{(1,\pm{1})} &\cong V_{(\mu-1,\pm(\mu-1))}\oplus V_{(\mu,\pm \mu)} \oplus V_{(\mu+1,\pm(\mu+1))},\\
	V_{(\mu,\pm \mu)} \otimes V_{(1,0)}&\cong V_{(\mu+1,\pm\mu)} \oplus V_{(\mu,\pm(\mu-1))}.
\end{align*}
Applying the Chern character yields the recurrence relations
\begin{align*}
	\mathrm{ch}(V_{(\mu+1,\pm(\mu+1))})&= \mathrm{ch}(V_{(\mu,\pm \mu)})(\mathrm{ch}(V_{(1,\pm 1)})-1)- \mathrm{ch}(V_{(\mu-1,\pm(\mu-1))})\\
	\mathrm{ch}(V_{(\mu+1,\pm\mu)}) &= \mathrm{ch}(V_{(\mu,\pm\mu)})\cdot \mathrm{ch}(V_{(1,0)})-\mathrm{ch}(V_{(\mu,\pm(\mu-1))})
\end{align*}
Writing $\mathrm{ch(V_{(1,\pm 1)})} = 3 + c_{1}^{\pm}$, where $c_{1}^{\pm}$ is the degree-$4$-part of the Chern character, we can solve these recursions obtaining
\begin{align*}
	\mathrm{ch}(V_{(\mu,\pm \mu)}) &= 1+2\mu +  \left( \frac{1}{6}\mu+\frac{1}{2}\mu^{2}+\frac{1}{3}\mu^{3} \right)c_{1}^{\pm},\\
	\mathrm{ch}(V_{(\mu+1,\pm \mu)})_{4} &= \frac{2}{3}(\mu^{3}+3\mu^{2}+2\mu)c_{1}^{\pm} + \left( \sum_{\lambda=0}^{\mu}(-1)^{\mu-\lambda}(1+2\lambda) \right)(\operatorname{ch}V_{(1,0)})_{4},
\end{align*}
where $\mathrm{ch}(V_{(\lambda_{1},\lambda_{2})})_{4}$ denotes the degree $4$-part of $\mathrm{ch}(V_{(\lambda_{1},\lambda_{2})})$.
An easy calculation involving the Atiyah--Singer index theorem (see Theorem \ref{thm:atiyahsinger}) shows that the index of the higher signature operator is given by
\begin{equation*}
	\operatorname{ind} P_{\mu}^{+}= \left(1+\mu\right) \operatorname{ind} P_{0}^{+} = \left( 1+\mu \right)L[M],
\end{equation*}
where in the last step we have used that the index of the signature operator is given by the $L$-genus $L[M]$ of $M$ (see \cite[Theorem 6.6]{atiyahIndexEllipticOperators1968}).

\section{The heat kernel expansion}\label{seq:hke}
Fix $G_{n}\in \{ \mathrm{SO}(n), \mathrm{Spin}(n) \}$ and let $H_{n}\in \{ \mathrm{O}(n), \mathrm{Pin}(n) \}$ be such that $G_{n}\subseteq H_{n}$. 
Let $(\sigma,\epsilon)$ be a chiral geometric symbol acting on an $H_{n}$-representation $V$ and $(M,P)$ be a $G_{n}$-manifold. Denote by $g$ the metric and by $\smallO$ the orientation of $M$. Let $\xi=(E,h, \nabla^{\xi})$ be a hermitian vector bundle in the sense of section \ref{sec:invarianttheory}. Consider the twisted geometric operator
$$
	D^{\xi}:C^\infty(M,VM \otimes E)\rightarrow C^\infty(M,VM\otimes E)
$$
obtained from $\sigma$ and $\xi$. Since $\nabla^{\xi}$ is a metric connection, $D^{\xi}$ is self-adjoint. The twisted chirality operator $\epsilon^{\xi}= \epsilon \otimes \mathrm{id}_E$ still is an involution which induces the splitting
$$
  VM \otimes E = V^{+}M\otimes E \oplus V^{-}M\otimes E.
$$
The twisted geometric operator $D^{\xi}$ is odd with respect to this splitting and hence the heat semigroup of $(D^{\xi})^{2}$ takes the form
$$\exp(-t(D^{\xi})^2)=\begin{pmatrix}
	\exp(-tD_-^{\xi}D_+^{\xi}) & 0 \\ 0 & \exp(-tD_+^{\xi}D_-^{\xi})
\end{pmatrix}.$$
It is a classical fact that the heat semigroup $\exp(-t(D^{\xi})^2)$ consists of smoothing operators. The index of $D_+^{\xi}$ is given by the formula
\begin{align*}
	\ind(D_{+}^{\xi}) &= \tr_{L^2} \exp(-tD_-^{\xi} D_+^{\xi}) - \tr_{L^2}\exp(-tD_+^{\xi} D_-^{\xi})\\
	&= \int_M \left(\tr \exp(-tD_-^{\xi} D_+^{\xi})(x,x) - \tr \exp(-tD_+^{\xi} D_-^{\xi})(x,x)\right) \mathrm{dvol}_{g,\smallO}\\
	&= \int_M \mathrm{str}(\exp(-t(D^{\xi})^2)(x,x))\mathrm{dvol}_{g, \smallO}, 
\end{align*}
where for $A\in VM \otimes E\otimes (VM\otimes E)^\ast$ we have set 
$$\mathrm{str}(A)=\mathrm{tr}(\epsilon(\mathrm{dvol}_{g,\smallO})A)$$
and where we integrate with respect to the integral of $n$-forms induced by the orientation $\smallO$ of $M$.

\begin{defin}
	A sequence of endomorphism-valued sections $\varPhi_k\in C^\infty(M, \mathrm{End}(VM\otimes E))$ are said to be an asymptotic expansion for $\exp(-t(D^{\xi})^2)$ if for all $m\in\mathbb{N}$, there exists $N\in \mathbb{N}$, such that 
	$$\lVert \exp(-t(D^{\xi})^2)(x,x)-\sum_{k\leq N}\varPhi_k(x)t^{\frac{k-n}{2}}\rVert_{C^k}\leq C t^m$$
	holds for all $t\in (0,1)$. Write
	$$\exp(-t(D^{\xi})^2)(x,x) \sim \sum_{k = 0}^{\infty}\varPhi_k(x)t^{\frac{k-n}{2}}.$$
\end{defin}

The construction of the asymptotic expansion of the heat kernel in this article is taken from Gilkey's book \cite{gilkeyInvarianceTheoryHeat1994}.
By Gilkey \cite[Lemma 1.8.2, (1.8.3), Lemma 1.7.2]{gilkeyInvarianceTheoryHeat1994}, the $\varPhi_k$ are constructed locally as follows:\\
In a chart of $M$ we can expand the symbol $\sigma(A)$ of $A$ as
$$
  \sigma((D^{\xi})^{2})(\xi)= \sum_{0 \leq k \leq m} a_{k}(\xi)
$$
where $a_{k}(\xi)$ is a homogeneous polynomial of degree $k$ in $\xi$ and defined by
$$
  (D^{\xi})^{2}= \sum_{\left\lvert \alpha\right\rvert \leq m} a_{\alpha}(x)D_{x}^{\alpha}, \quad a_{k}(\xi)= \sum_{\left\lvert \alpha\right\rvert=k} a_{\alpha}(x)\cdot \xi^{\alpha},
$$
where  $D_{x}^{\alpha}= (-i)^{\left\lvert \alpha\right\rvert} \frac{ \partial^{\alpha}  }{\partial x^{\alpha} } $.
For $\lambda\in \mathbb{C}\setminus\{x\in\mathbb{R}\mid x\geq 0\}$, we approximate the symbol of the parametrix $((D^{\xi})^2-\lambda)^{-1}$ by inverting the symbol of $(D^{\xi})^2-\lambda$ formally using the following series:
\begin{subequations}
	\begin{equation}\label{eq:recsymbol1}
		b_{-2}(x,\xi,\lambda) = (a_2(x,\xi)-\lambda)^{-1},
	\end{equation}
	\begin{equation}\label{eq:recsymbol2}
		b_{-k-2}(x,\xi,\lambda) = -\left(\sum_{\substack{
				\lvert\alpha\rvert+j+l=k\\
				j< k
		}} \frac{1}{\alpha !}\frac{\partial^\alpha b_{-2-j}}{\partial \xi^\alpha}(x,\xi,\lambda)\cdot D_x^\alpha a_{2-l}(x,\xi)\right)\cdot b_{-2}(x,\xi,\lambda).
	\end{equation}
\end{subequations}
In order to obtain the asymptotic expansion $\varPhi_k$ for the heat kernel, we integrate clockwise over a contour $\Gamma \subseteq \mathbb{C}$ enclosing the nonnegative real axis and subsequently integrate over $\xi$:
\begin{align}\label{eq:asymptotics}
		\varPhi_k(x) = \sqrt{\det g(x)}^{-1}\frac{1}{2\pi i}\iint_\Gamma e^{-\lambda}b_{-2-k
		}(x,\xi,\lambda)d\lambda d\xi.
\end{align}
The meromorphicity  of $b_k(x,\xi,\lambda)$ in $\lambda$ implies that $\varPhi_{k}$ does not depend on the contour $\Gamma$. Athough $\varPhi$ is locally constructed, it glues to a well-defined section of the endomorphism bundle of $VM\otimes E$.
We obtain the following asymptotic expansion of the heat supertrace:
\begin{align*}
	\mathrm{str}(\exp(-tD^2)(x,x))\mathrm{dvol}_{g,\smallO} \sim & \sum_{k=0}^\infty\left(\mathrm{str}(\varPhi_k(x))\mathrm{dvol}_{g,\smallO}\right) t^{\frac{k-n}{2}}.
\end{align*}

\begin{defin}
	Let $(\sigma,\epsilon)$ be a chiral geometric symbol, $(M,P)$ a $G_{n}$-manifold and $\xi$ a hermitian vector bundle over $M$. The $k$-th heat form of the corresponding twisted geometric operator is given by
	\begin{equation}\label{eq:heatcoefficient}
		\omega_k(P, \xi)(x)= \mathrm{str}(\varPhi_k(x))\mathrm{dvol}_{g,\smallO}.
	\end{equation}
\end{defin}

\begin{rmk}

	\begin{itemize}
		\item For $n= \dim M$, the $n$-th heat coefficient $\omega_{n}$ has the property
		\begin{equation}\label{eq:indexdensity}
			\mathrm{ind}(D_+^{\xi})= \int_M \omega_n(P,\xi).
		\end{equation}
		For this reason, $\omega_n$ is also referred to as the index density of $D_+^{\xi}$.
		\item For a local diffeomorphism $f: N \to M$ we have the relation $$
		  \omega_{k}( f^{*}P, f^{*}\xi) = f^{*}\omega_{k}(P, \xi).
		$$
		Thus, $\omega_{k}$ defines a natural transformation
		$$
		  \omega_{k} : \mathrm{Str}^{G_{n}} \times \mathrm{Vect}_\mathbb{C}^{m} \to \Omega^{n}.
		$$
	\end{itemize}
\end{rmk}

\begin{prop}\label{prop:nattran}
	Let $(\sigma, \epsilon)$ be a chiral geometric symbol. Let $M\in \mathbf{Man}_{n}^{G_{n}}$ and let $\xi$ be a hermitian vector bundle over $M$. Then if $P$, $P^{\prime}$ are two $G_{n}$-structures inducing the same Riemannian metric $g$ on $M$, the heat forms $\omega_k(P,\xi)$ and $\omega_{k}(P^{\prime},\xi)$ concide. 
  	The natural transformation $\omega_{k}$ factors through a natural transformation
	$$\omega_k\colon \mathrm{Met}\times \mathrm{Vect}_{\mathbb{C}}^{m}\rightarrow \mathrm{\Omega^n}.$$
\end{prop}

\begin{proof}
	We first consider the untwisted case.
	Let $U$ be a contractible open manifold with metric $g$. Let $P_{1}$ and $P_{2}$ be two $G_{n}$-structures inducing the metric $g$ and let $P_{i}^{H_{n}}=P_{i}\times_{G_{n}}H_{n}$ be the induced $H_{n}$-structures. If $\mathrm{O}_{g}(U)$ is the orthonormal frame bundle of $U$, we have a covering $\pi_{i}:P_{i}^{H_{n}}\to \mathrm{O}_{g}(U)$ induced by the identification of the tangent bundle $TU$ with $P\times_{G_{n}}\mathbb{R}^{n}$. Since $U$ is contractible, we can choose an orthonormal frame $(b_{1},\dots,b_{n})$ of $TU$, lifting these to sections $s_{i}:U\to P_{i}$ gives us a commutative diagram
	\[
	\begin{tikzcd}
		P_1^{H_n} \arrow[r, "\pi_1"] & O_g(U)                                                                          & P_2^{H_n}  \arrow[l, "\pi_2"] \\
									 & U \arrow[u, "{(b_1,...,b_n)}" description] \arrow[ru, "s_2"'] \arrow[lu, "s_1"] &                              
	\end{tikzcd}.
	\]
	This induces an isomorphism of $H_{n}$-bundles
	$$
		s_{21}:P_{1}^{H_{n}}\to P_{2}^{H_{n}}, \quad s_{1}(x)\cdot h\mapsto s_{2}(x)\cdot h
	$$
	such that 
	\[
	\begin{tikzcd}
		P_1^{H_n} \arrow[rr, "s_{21}"] \arrow[rd, "\pi_1"'] &                 & P_2^{H_n} \arrow[ld, "\pi_2"] \\
													& \mathrm{O}_g(U) &                             
	\end{tikzcd}
	\]
	commutes. For the connection form we have
	\begin{align*}
		s_{21}^{*}\omega^{LC}_{2} & =s_{21}^{*}\rho_{*}^{-1}\pi_{2}^{*}\omega^{LC} \\
 			& = \rho_{*}^{-1}(s_{21}^{*}\pi_{2}^{*}\omega^{LC}) \\
 			& = \rho_{*}^{-1} \pi_{1}^{*}\omega^{LC} = \omega_{1}^{LC},
	\end{align*}
	where $\rho_{*}: \mathfrak{g}_n \to \mathfrak{so}_{n}$ is the Lie algebra isomorphism induced by either the identity in the case that $G_{n}$ is $\mathrm{SO}(n)$ or the usual covering $\mathrm{Spin}(n)\to \mathrm{SO}(n)$ in the case that $G_{n}$ equals $\mathrm{Spin}(n)$.
	This shows that the induced vector bundle isomorphism 
	$$
		s_{21}^{V}:V_{1}U\to V_{2}U \text{ is parallel, i.e. }\nabla^{2}\circ s_{21}^{V} = s_{21}^{V}\circ \nabla^{1}.
	$$
	Moreover, for the symbol $\sigma$ we have
	$$
		\sigma_{2}\circ s_{21}^{V\otimes \mathbb{R}^{n}}= s_{21}^{V}\circ \sigma_{1},
	$$
	and thus $D_{2}\circ s_{21}^{V} =s_{21}^{V}\circ D_{1}$. Therefore, the heat asymptotics satisfy $\varPhi^{2}_{k}\circ s_{21}^{V}=s_{21}^{V}\circ \varPhi_{k}^{1}$. The same holds if we twist $D_{1}$ and $D_{2}$ with the hermitian bundle $\xi$.
	Similarly, we have for $\epsilon_{i}=[p_{i},\epsilon]\in P^{H_{n}}_{i}\times_{H_{n}}(\operatorname{End}V\otimes \Lambda^{n}\mathbb{R}^{n})$
	$$
		\epsilon_{2}\circ s_{21}^{V\otimes \Lambda^{n}\mathbb{R}^{n}} = s_{21}^{V}\circ \epsilon_{1}.
	$$
	Now there are two cases. The first case is $s_{21}(P_{1})=P_{2}$, then we have that for $p \in P_{1}$
	\begin{align*}
		\mathrm{dvol}_{1} & = (\Lambda^{n}\phi_{1})([p,e_{1}\wedge\dots \wedge e_{n}]) \\
 			& = (\Lambda^{n}\phi_{2})([s_{21}(p),e_{1}\wedge\dots \wedge e_{n}]) \\
 			& = \mathrm{dvol}_{2},
	\end{align*}
	where $\phi_{i}$ is the homomorphism identifying $P_{i}\times_{G_{n}} \mathbb{R}^{n}$ with the tangent bundle $TU$.
	If $s_{21}(P_{1})\neq P_{2}$, then for any $h\in H_{n}\setminus G_{n}$ $s_{21}(P_{1})\cdot h=s_{21}(P_{2})$. For the volume form this implies
	$$
		\mathrm{dvol}_{1}=-\mathrm{dvol}_{2}.
	$$
	Thus
	\begin{align*}
		\omega_{k}(P_{1},\xi) &  = \operatorname{tr}(\epsilon_{1}(\mathrm{dvol}_{1})\cdot \varPhi_{k}^{1})\mathrm{dvol}_{1} \\
 			& = \operatorname{tr}(\epsilon_{1}(\mathrm{dvol_{2}})\cdot \varPhi^{1}_{k})\mathrm{dvol_{2}} \\
 			& =\operatorname{tr}(s_{21}^{-1}\cdot\epsilon_{2}(\mathrm{dvol_{2}})\cdot \varPhi^{2}_{k}\cdot s_{21})\mathrm{dvol_{2}} \\
 			& = \operatorname{tr}(\epsilon_{2}(\mathrm{dvol_{2}})\cdot \varPhi^{2}_{k})\mathrm{dvol_{2}} \\
 			& = \omega_{k}(P_{2},\xi).   
	\end{align*}
	Therefore, the first part of the proposition is shown.

	For the second part, given any Riemannian manifold $(M,g)$ equipped with a hermitian vector bundle $\xi$, we can construct $\omega_{k}(g,\xi)$ locally in coordinates, the arguments above show that the construction is invariant under coordinate changes. From this follows also the naturality of $\omega_{k}(g,\xi)$, showing that it defines a natural transformation.
\end{proof}

\begin{prop}\label{prop:jointinvariant}
	Let $(\sigma,\epsilon)$ be a chiral geometric symbol. Then the natural transformation $\omega_{k}: \mathrm{Met}\times \mathrm{Vect}^{m}_\mathbb{C}\to \Omega_{n}$, defined in Equation \ref{eq:heatcoefficient}, is a joint geometric invariant.
\end{prop}

What remains to show is that $\omega_{k}(g,\xi)$ is, when expressed in a normal trivialization around a point $p$, given by a universal polynomial expression in the curvature tensors of $g$ and $\xi$ and its covariant derivatives. In order to achieve this, one first sees that the total symbol of $D^{\xi}$ only depends on the metric and $\Gamma^{\xi}$, this property then transfers to summands $b_{-k}$ in the asymptotic expansion of the symbol of the resolvent of $(D^{\xi})^{2}-\lambda$. As it turns out, the integrals in the definition \ref{eq:asymptotics} of $\Phi_{k}$ will not affect this either, yielding the statement after taking the trace.
  
\begin{proof}
	Let $(M,g)$ be a Riemannian manifold and $\xi$ be a Hermitian vector bundle over $M$. Let $p \in M$ and choose a normal trivialization centered at $p$, and let $A$ be the transition matrix defined by Equation \ref{eq:framechange}. Then as shown in the proof of Proposition \ref{prop:nattran}, we can without loss of generality assume that $M$ carries a $H_{n}$-structure $P$. Denote by $D^{\xi}$ the induced twisted geometric operator.
  
	With respect to the normal trivialisation, write
	$$
	  (D^{\xi})^{2} = \sum_{\lvert \alpha \rvert \leq 2} a_{\alpha}(x)D_{x}^{\alpha}, \quad a_{k}(x,\zeta) = \sum_{\lvert \alpha \rvert =k}a_{\alpha}(x)\zeta^{\alpha}
	$$
	From Lemma \ref{lem:opincoord}, we see that the $a_{\alpha}(x)$ are given by polynomials in $A^{i}_{j}$, $(A^{-1})^{i}_{j}$, $\Gamma^{g}$, $\Gamma^{\xi}$ and their derivatives. In particular, the principal symbol of $(D^{\xi})^{2}$ is given by
	$$
	  a_{2}(x,\zeta) = - (\sigma^{j}\otimes \mathrm{id})(\sigma^{l}\otimes \mathrm{id})A^{i}_{j}(x)A^{k}_{l}(x)\zeta_{i}\zeta_{l}
	$$
	Since $A(0)=\mathrm{id}$, we have 
	$$
	  a_{2}(0,\zeta)= -(\sigma^{j}\otimes \mathrm{id})(\sigma^{l}\otimes \mathrm{id})\zeta_{j}\zeta_{l}.
	$$
	By Corollary \ref{cor:taylor}, the coefficients $a_{\alpha}(x)$ and all their derivatives are given by polynomials in the covariant derivatives of $R_{ijkl}$, $K^{\mu}_{\nu kl}$ when evaluated at $x=0$.
  
	Define $b_{-m-k}$ by the recursion as in \ref{eq:recsymbol1}, \ref{eq:recsymbol2}, 
	\begin{equation*}
	  b_{-2}(x,\zeta,\lambda) = (a_2(x,\zeta)-\lambda)^{-1},
	\end{equation*}
	\begin{equation*}
	  b_{-2-k}(x,\zeta,\lambda) = -\Big(\sum_{\substack{
	  \lvert\alpha\rvert+j+l=n\\
	  j< k
	  }} \frac{1}{\alpha !}\partial_{\zeta}^\alpha b_{-2-j}(x,\zeta,\lambda)\cdot D_x^\alpha a_{2-l}(x,\zeta)\Big)\cdot b_{-2}(x,\zeta,\lambda).
	\end{equation*}
	Since $D_{x}^{\alpha}a_{2-l}(0,\zeta)$ is a polynomial in the covariant derivatives of the curvature tensors $R_{ijkl}$, $K^{\mu}_{\nu kl}$ at $0$, we see that if  $b_{-2-j}(0,\zeta,\lambda)$ is a polynomial in the covariant derivatives of $R_{ijkl}$, $K^{\mu}_{\nu kl}$ for $0\leq j< k$, then so is $b_{-2-k}(0,\zeta,\lambda)$.  Here, the coefficients of the polynomials are matrix-valued functions of $(\zeta,\lambda)$, however they do not depend on the chosen normal trivialization. 
	Since 
	$$
	  b_{-2}(0,\zeta,\lambda) = -((\sigma^{j}\otimes \mathrm{id})(\sigma^{l}\otimes \mathrm{id})\zeta_{j}\zeta_{l} +\lambda)^{-1}
	$$
	is completely independent of the metric, we obtain inductively that $b_{-2-k}(0,\zeta,\lambda)$ is indeed given by a polynomial in the covariant derivatives of $R_{ijkl}$, $K^{\mu}_{\nu kl}$, with the coefficients being rational functions in $(\zeta,\lambda)$.
	Then the integrals in Equation \ref{eq:asymptotics} of the heat asymptotics
	$$
	\varPhi_{k}(0) = \frac{1}{2\pi i} \iint_{\Gamma}e^{ -\lambda } b_{-2-k}(0,\zeta,\lambda)d\lambda d\zeta. 
	$$
	only affect the coefficients, and thus the dependence of $\varPhi_{k}(0)$ on the covariant derivatives of $R_{ijkl}$, $K^{\mu}_{\nu kl}$ remains polynomial. 
	Then, the local index coefficients are given by
	$$
	  \omega_{k}(g,\xi)(0)= \operatorname{tr}(\epsilon(\mathbf{e}_{1}\wedge\dots \wedge \mathbf{e}_{n})\otimes \mathrm{id} \cdot e_{k}(0)) \cdot\mathbf{e}_{1}\wedge\dots \wedge \mathbf{e}_{n},
	$$
	where $\mathbf{e}_{1},\dots,\mathbf{e}_{n}$ is the standard basis of $\mathbb{R}^{n}$. Hence, $\omega_{k}(g)(0)$ too, is given by a polynomial in the covariant derivatives of $R_{ijkl}$, $K^{\mu}_{\nu kl}$.
\end{proof}

\begin{prop}\label{prop:homogeneity}
	Let $(\sigma,\varepsilon)$ be a chiral geometric symbol acting on a $H_{n}$-representation $V$.
	Then $\omega_k$ is homogeneneous of weight $(n-k, 0)$ in $(g, \xi)$, i.e.
	$$\omega_k(\lambda^2 g, \mu^{2}\xi)= \lambda^{n-k}\omega_k(g, \xi), \quad \forall \lambda,\mu> 0.$$
\end{prop}

\begin{proof}
	Fix $\lambda,\mu > 0$. 

	Let $(M,g)$ be a Riemannian manifold and $\xi$ be a Hermitian vector bundle over $M$. W.l.o.g., we can assume that the metric is induced by a $G_{n}$-structure $(P,\phi)$ over $M$. The $G_{n}$-structure $(P, \lambda ^{-1} \phi)$ then induces the metric $\lambda^{2}g$. Let $D$, $D_{\lambda}$ be the induced geometric operators acting on sections of the vector bundle $VM = P\times_{G_{n}}V$. 

	Since the Levi-Civita connection is invariant under the metric change $g \mapsto \lambda^{2}g$, the induced connection $1$-forms on $P$ coincide. Thus both $G_{n}$-structures induce the same covariant derivative $\nabla$ on $VM$. 

	Denote by $\bar{\sigma}$, $\bar{\sigma}_{\lambda}$ the by $\sigma$ induced sections of $\mathrm{End}(VM)\otimes TM$. Let $p \in P$, then $\sigma_{\lambda}$ is given by
	\begin{align*}
		\bar{\sigma}_{\lambda}&= (\mathrm{id}_{\mathrm{End}(VM)}\otimes \phi_{\lambda})([p,\sigma]) \\
			&= \frac{1}{\lambda} (\mathrm{id}_{\mathrm{End}(VM)}\otimes \phi)([p,\sigma]) = \frac{1}{\lambda} \bar{\sigma}.
	\end{align*}
	Since $D$, $D_{\lambda}$ are defined as the contraction of $\sigma$, $\sigma_{\lambda}$ with $\nabla$, the identity
	$$
	  D_{\lambda} = \frac{1}{\lambda} D
	$$
	holds.
	Rescaling the hermitian bundle $\xi$ with $\mu$ does not affect the covariant derivative $\nabla^{\xi}$ of $\xi$, and thus we obtain the same equality for the twisted geometric operators $D^{\xi}_{\lambda}= \frac{1}{\lambda}D^{\xi}$. 

	Since $D^{\xi}$ is self-adjoint, there exists an orthonormal eigenbasis $\phi_{n}$ of $L^{2}(M;VM\otimes E^{\xi})$ with eigenvalues $\nu_{n}\in \mathbb{R}$.
	We have
	  $$\sum_{n\in \mathbb{N}} e^{-\nu_n^2 t} \mathrm{str}(\phi_n\otimes \phi_n^*) \mathrm{dvol}_g=\operatorname{str} \exp(-t(D^{\xi})^2)(x,x) \mathrm{dvol}_g \sim \sum \omega_k(g,\xi)(x)t^{\frac{k-n}{2}}.$$
	The spectrum of $D_{\lambda}^{\xi}$ is given by $(\frac{\nu_{n}}{\lambda})_{n\in \mathbb{N}}$ with eigenbasis $(\phi_{n})_{n\in \mathbb{N}}$. We now argue as Atiyah, Bott and Patodi in \cite{atiyahHeatEquationIndex1973}:
	  \begin{align*}
		  \sum \omega_k(\lambda^2 g, \mu^{2}\xi)t^{\frac{k-n}{2}}&\sim
		   \sum_{n\in \mathbb{N}} e^{-\frac{\nu_n^2}{\lambda^2} t} \operatorname{str}(\phi_n\otimes \phi_n^*) \cdot \mathrm{dvol}_g\\
		  &\sim \sum \omega_k(g,\xi)\left(\frac{t}{\lambda^2}\right)^{\frac{k-n}{2}}\\ 
		  &= \sum  \lambda^{n-k}\omega_k(g,\xi)t^{\frac{k-n}{2}}.
	  \end{align*}
	  The claim follows from the uniqueness of asymptotic expansions.
\end{proof}

Summarizing the results of Proposition \ref{prop:nattran}, Proposition \ref{prop:jointinvariant} and Proposition \ref{prop:homogeneity} and applying Gilkey's Theorem \ref{thm:Gilkey} we obtain the following preliminary version of the local index theorem:

\begin{thm}\label{thm:indexcoefficients}
	Let $(\sigma,\varepsilon)$ be a chiral geometric symbol and $\omega_k$ be the associated heat forms. Let $(M,P)$ be a $G_{n}$-manifold with Riemannian metric $g$, $\xi$ a hermitian vector bundle over $M$, and $D^{\xi}$ be the geometric operator associated to $\sigma$ twisted with $\xi$. Then for $k<n$, the heat forms $\omega_k(g,\xi)$ are zero, and $\omega_n(g,\xi)$ is given by a polynomial in the Pontryagin-forms of $g$ and the Chern-forms of $\xi$. In particular, $\mathrm{str}\left(\exp(-t(D^{\xi})^2)(x,x)\right)\mathrm{dvol}_g$ converges for $t\searrow 0$ with
	$$\lim_{t\searrow 0}\mathrm{str}\left(\exp(-t(D^{\xi})^2)(x,x)\right)\mathrm{dvol}_g = \omega_n(g,\xi)(x).$$
\end{thm}

\section{The local index theorem}\label{section:localindex}

Let $G_n\in \{\mathrm{SO}(n),\mathrm{Spin}(n)\}$ for $n=2l$ even and let $(\sigma,\varepsilon)$ be a chiral geometric symbol acting on a $G_{n}$-representation $V$. We want to identify the local index density $\omega_{n}(g,\xi)$ of the induced chiral geometric operator. We will follow the strategy pursued by Gilkey \cite{gilkeyInvarianceTheoryHeat1994} and Atiyah, Bott and Patodi \cite{atiyahHeatEquationIndex1973}.

Let $(M,P)$ be a $G_{n}$-manifold, $\xi=(E,h,\nabla^{\xi})$ be a hermitian vector bundle over $M$ with fiber dimension $m$ and let $D^{\xi}$ be the induced chiral geometric operator on $M$. Let $\mathrm{U}(E)$ be the orthonormal frame bundle of $E$, it is a principal $\mathrm{U}(m)$-bundle over $M$, such that
\begin{equation*}
  \mathrm{U}(E)\times_{\mathrm{U}(m)}\mathbb{C}^{m} \cong E.
\end{equation*} 
Let $\tilde{P}$ be the fiber product of $\pi_{1}\colon P \to M$ and $\pi_{2}\colon U(E) \to M$, i.e.
\begin{equation*}
  \tilde{P} = \{ (p,h)\in P\times \mathrm{U}(E) \mid \pi_{1}(p) = \pi_{2}(h) \}.
\end{equation*}
This is a principal $G_{n}\times \mathrm{U}(m)$-bundle over $M$, such that 
\begin{align*}
  \tilde{P}\times_{G_{n}\times \mathrm{U}(m)}\mathbb{R}^{n}&\cong TM,\\
  \tilde{P}\times_{G_{n}\times \mathrm{U}(m)} V\times \mathbb{C}^{m} &\cong VM\otimes E,
\end{align*}
where $\mathrm{U}(m)$ acts trivially on $\mathbb{R}^{n}$. In particular, $\tilde{P}$ is a $G_{n}\times U(m)$-structure on $M$ in the sense of Definition \ref{def:groupstructure}. Now observe that the twisted symbol
$$
  \sigma^{+}\otimes \mathrm{id}_{\mathbb{C}^{m}}\colon V^{+}\otimes \mathbb{C}^{m}\otimes \mathbb{R}^{n}\to V^{-}\otimes \mathbb{C}^{m}
$$
is $G_{n}\times \mathrm{U}(m)$-equivariant and hence gives a universal symbol class 
$$
\sigma^{+}\otimes \mathrm{id}_{\mathbb{C}^{m}}\in K_{G_{n}\times \mathrm{U}(m)}(\mathbb{R}^{n})
$$
in the equivariant $K$-theory of $\mathbb{R}^{n}$. This allows us to calculate the index of $D_{+}^{\xi}$ using the Atiyah--Singer index theorem for elliptic complexes associated to group structures \cite[Proposition 2.17]{atiyahIndexEllipticOperators1968}. Denote by $\chi(F)$ the Euler class of a real oriented vector bundle.

\begin{thm}[Atiyah--Singer]\label{thm:atiyahsinger}
	Let $(\sigma,\varepsilon)$ be a chiral $G_n$-geometric symbol for $n=2l$ even, acting on a $G_{n}$-representation $V$. Let $(M,P)$ be a $G_n$-manifold with induced orientation $\smallO$, and let $\xi$ be a hermitian vector bundle over $M$. Let 
  $D^{\xi}$ be the induced twisted geometric operator and $D^{\xi}_{+}$ be its positive chiral part. Then the characteristic class
  \begin{equation*}
    \frac{\mathrm{ch}(V_{+}M)-\mathrm{ch}(V_{-}M)}{\chi(TM_{\smallO}) }\in H^*(M,\mathbb{R})
  \end{equation*}
  is well-defined and
	\begin{equation}\label{eq:asindex}
		\mathrm{ind}(D_{+}^{\xi})= (-1)^l\left\langle\mathrm{ch}(E^{\xi})\cdot \frac{\mathrm{ch}(V_{+}M)-\mathrm{ch}(V_{-}M)}{\chi(TM_{\smallO})}\cdot \hat{A}(M)^2,[M_{\smallO}]\right\rangle,
	\end{equation}
	where $\left\langle \alpha, [M]_{\smallO}\right\rangle$ denotes the pairing of a cohomology class $\alpha$ with the fundamtenal class $[M]_{\smallO}$ induced by the orientation $\smallO$ of $M$.
\end{thm}

\begin{example}
	Recall the higher Dirac operators $D_{j}$ defined in Remark \ref{rmk:higher_Dirac_operator}. Bure\v{s} \cite{buresHigherSpinDirac1999} computed their index, yielding
	\begin{equation}\label{eq:indexhigherdirac}
		\ind D_{j,+} = \langle(\mathrm{ch}(\Lambda^j T^\ast_\mathbb{C}M)+ \mathrm{ch}(\Lambda^{j-1}T^\ast_\mathbb{C}M))\hat{\mathrm{A}}(M),[M]\rangle.
	\end{equation}
	In particular, for the Rarita--Schwinger operator $Q=D_{1}$ we obtain following formula for the index
	\begin{equation}\label{eq:indexrs}
		\ind Q_+ = \langle \hat{\mathrm{A}}(TM)(\mathrm{ch}(T_\mathbb{C}M)+1), [M]\rangle
	\end{equation}
	see also \cite{hommaKernelRaritaSchwinger2019, barManifoldsManyRarita2021}.
\end{example}

Since 
\begin{equation}\label{eq:characteristicclass}
  P\mapsto (-1)^{l}\cdot\frac{\mathrm{ch}(P\times_{G_{n}}V_{+})-\mathrm{ch}(P\times_{G_{n}} V_{-})}{\chi(P\times_{G_{n}}\mathbb{R}^{n})}\hat{A}(P\times_{G_{n}}\mathbb{R}^{n})^{2}
\end{equation}
is a natural transformation from principal $G_{n}$-bundles to real cohomology, it thus defines a real characteristic class of $G_{n}$. It is therefore represented by some $G_{n}$-invariant polynomial on the Lie-Algebra $\mathfrak{g}_{n}$. Since the adjoint representation of $\mathrm{Spin}(n)$ factors through $\mathrm{SO}(n)$, and the Lie algebras of $\mathrm{SO}(n)$ and $\mathrm{Spin}(n)$ are isomorphic (as $\mathrm{SO}(n)$-representations) the real characteristic classes of $G_{n}$ are exactly those of $\mathrm{SO}(n)$.
Therefore the characteristic class \ref{eq:characteristicclass} is represented by some polynomial $$f(p_{1},\dots,p_{l-1},\chi)$$ in the Pontryagin classes and the Euler class. For a $G_{n}$-manifold $(M,P)$ with induced orientation $\smallO$ we thus have
\begin{equation*}
  \frac{\mathrm{ch}(V_{+,\smallO}M)-\mathrm{ch}(V_{-,\smallO}M)}{\chi(M_{\smallO})}\hat{A}(M)^{2}= f(p_{1}(M),\dots,p_{l-1}(M),\chi(M_{\smallO})).
\end{equation*}
For the reversed orientation $\bar{\smallO}$ we have by Proposition \ref{prop:chirality}
\begin{align*}
  \mathrm{ch}(V_{+,\bar{\smallO}}M)&= \mathrm{ch}(V_{-,\smallO}M)\\
  \mathrm{ch}(V_{-,\bar{\smallO}}M)&= \mathrm{ch}(V_{+,\smallO}M)
\end{align*}
and by the properties of the Euler-class
\begin{equation*}
  \chi(M_{\bar{\smallO}}) = -\chi(M_{\smallO}),
\end{equation*}
the $\hat{A}$-polynomial is known to be given by a polynomial in the Pontryagin-classes and is therefore invariant under changes of orientation.
Thus, the characteristic class \ref{eq:characteristicclass} is invariant under changes of orientation, implying 
\begin{equation*}
  f(p_{1},\dots,p_{l-1},\chi) = f(p_{1},\dots,p_{l-1},-\chi).
\end{equation*}
Therefore, the variable $\chi$ appears only as a square in $f$, and since $\chi^{2}=p_{l}$,
\begin{equation*}
  f(p_{1},\dots,p_{l-1}, \chi)= \tilde{f}(p_{1},\dots,p_{l})
\end{equation*}
for some polynomial $\tilde{f}$. 

Let us return to the situation where $(M,P)$ is a $G_{n}$-manifold with induced metric $g$ and $\xi=(E,h,\nabla^{\xi})$ is a hermitian vector bundle over $M$. Denote by $\Omega^{g}\in \Omega^{2}(P,\mathfrak{g}_n)$ the curvature $2$-form of $g$ and let $\Omega^{\xi}\in \Omega^{2}(\mathrm{U}(E),\mathfrak{u}(m))$ be the curvature $2$-form of the connection $\nabla^{\xi}$. The Chern--Weil homomorphism asserts the equality of the cohomology classes
\begin{align*}
  p_{i}(M)&=  \left[\sigma_{2i}\left(\frac{\Omega^{g}}{2\pi}\right)\right] \in H^{4i}(M,\mathbb{R}),\\
  \mathrm{ch}(E) &= \sum_{k} \frac{1}{k!}\cdot \left[s_{k}\left(\frac{\Omega^{\xi}}{2\pi i}\right)\right] \in H^{2*}(M,\mathbb{R}),
\end{align*}
where $\sigma_{i}$ is the $i$th elementary symmetric polynomial and $s_{k}$ is the $k$th power sum symmetric polynomial. The square brackets denote the de Rham cohomology class of closed differential forms. Rewriting
\begin{align*}
  p_{i}(g) = \sigma_{2i}\left(\frac{\Omega^{g}}{2\pi}\right), \quad \mathrm{ch}_{k}(\xi) = \frac{1}{k!} s_{k}\left(\frac{\Omega^{\xi}}{2\pi i}\right), 
\end{align*}
it follows from the above discussion that the de Rham cohomology class of the mixed-degree differential form
\begin{equation}\label{eq:polynomialincharacteristicclasses}
  \tilde{f}(p_{1}(g),\dots,p_{l}(g))\sum_{k}\mathrm{ch}_k(\xi)
\end{equation}
is equal to
$$
  (-1)^l\cdot\mathrm{ch}(E^{\xi})\cdot \frac{\mathrm{ch}(V_{+}M)-\mathrm{ch}(V_{-}M)}{\chi(TM_{\smallO})}\cdot \hat{A}(M)^2.
$$
The $n$-form part of the differential form is given by
\begin{equation}\label{eq:difffrom1}
  \sum_{k} \mathrm{ch}_k(\xi)\cdot \sum_{\substack{I\\\text{partitions of }\\ \frac{1}{4}(n-2k)}}a_{k,I}\cdot p_{I}(g)
\end{equation}
where $p_{I}= p_{i_{1}}\wedge\dots \wedge p_{i_{j}}$ and $a_{k,I}\in \mathbb{R}$. By the Atiyah--Singer index theorem \ref{thm:atiyahsinger}, it has the property that for any $G_{n}$-manifold $(M,P)$ equipped with an hermitian vector bundle $\xi$ the index of the induced operator $D^{\xi}_{+}$ is given by
\begin{equation}\label{eq:indexcohomology}
  \operatorname{ind}D_{+}^{\xi} = \int_{M} \sum_{k} \mathrm{ch}_k(\xi)\cdot \sum_{\substack{I\\\text{partitions of }\\ \frac{1}{4}(n-2k)}}a_{k,I}\cdot p_{I}(g),
\end{equation}
where $g$ is the induced metric and we integrate with respect to the induced orientation on $M$. 

Let $\omega_{n}(g,\xi)$ be the local index density of $D^{\xi}$, where we twist with $m$-dimensional vector bundles. By Theorem \ref{thm:indexcoefficients} there exists a polynomial $q_{m}$ such that
\begin{equation*}
  \omega_{n}(g,\xi) = q_{m}(\mathrm{ch}_{1}(\xi),..,\mathrm{ch}_{l}(\xi), p_{1}(g),\dots, p_{l}(g)).
\end{equation*}
Since $D^{\xi \oplus \eta}= D^{\xi}\oplus D^{\eta}$ and thus
$$
\operatorname{str}\exp(-t(D^{\xi \oplus \eta})^{2})= \operatorname{str}\exp(-t(D^{\xi})^{2})+ \operatorname{str}\exp(-t(D^{\eta})^{2}),
$$
it follows that $\omega_{n}$ is additive under Whitney sums, i.e.
\begin{equation*}
  \omega_{n}(g,\xi \oplus \eta) = \omega_{n}(g,\xi)+ \omega_{n}(g,\eta). 
\end{equation*}
Rewrite 
\begin{equation*}
  \omega_{n} = q_{m,0}(p) +q_{m,1}(\mathrm{ch},p)+\dots+ q_{m,l}(\mathrm{ch},p), 
\end{equation*}
where $q_{m,i}$ is homogeneous of degree $i$ in the variables $(\mathrm{ch}_k)_{k\geq 1}$. The additivity of $\omega_{n}$ implies for $k\geq 0$ 
$$
  \omega_{n}(g,\xi) = \omega_{n}(g, \xi \oplus \underline{\mathbb{C}}^{k})-  \frac{k}{m+k}\omega_{n}(g, \underline{\mathbb{C}}^{m+k}),
$$
$\underline{\mathbb{C}}^{k}$ being the trivial $k$-dimensional bundle. From this, we easily obtain
\begin{align*}
  q_{m,0}(p)&= \frac{m}{m+k}q_{m+k,0}(p),\\
  q_{m,i}(\mathrm{ch},p)&= q_{m+k,i}(\mathrm{ch},p) \text{ for }i\geq 1.
\end{align*}
Therefore the polynomials $(q_{m,i})_{i\geq 1}$ are independent of the dimension $m$ of the twist bundle, and we can rewrite $\omega_{n}$
$$
  \omega_{n}(g,\xi)= \mathrm{ch}_{0}(\xi) q_{0}(p(g))+ \sum_{i\geq 1}q_{i}(\mathrm{ch}(\xi),p(g)).
$$
For any hermitian vector bundle $\xi$ and $N\in \mathbb{N}$ the additivity of $\omega_{n}$ and the Chern-character imply
\begin{align*}
  N\Big(\mathrm{ch}_{0}(\xi) q_{0}(p(g))+ \sum_{i\geq 1}q_{i}\left(\mathrm{ch}(\xi),p(g)\right)\Big) &= N\omega(g,\xi)\\
  &= \omega(g,N\xi)\\
  &= \mathrm{ch}_{0}(N\xi) q_{0}(p(g))+ \sum_{i\geq 1}q_{i}(\mathrm{ch}(N\xi),p(g))\\
  &=  N\mathrm{ch}_{0}(\xi) q_{0}(p(g))+ \sum_{i\geq 1}N^{i}q_{i}(\mathrm{ch}(\xi),p(g)).
\end{align*}
Thus $q_{i}=0$ for $i\geq 2$ and $\omega_{n}$ is linear in the Chern-character,
\begin{equation}\label{eq:diffform2}
  \omega_{n}(g,\xi) = \sum_{k=0}^{l}\mathrm{ch}_{k}(\xi)\cdot \sum_{\substack{I\\\text{partitions of }\\ \frac{1}{4}(n-2k)}}b_{k,I}\cdot p_{I}(g),
\end{equation}
for some $b_{k,I}\in \mathbb{R}$. By equations \ref{eq:indexdensity} and \ref{eq:indexcohomology}, for each $G_{n}$-manifold $(M,P)$ and each hermitian vector bundle $\xi$ over $M$ the equality
\begin{equation}\label{eq:equalityofforms}
  \sum_{k=0}^{l} \sum_{\substack{I\\\text{partitions of }\\ \frac{1}{4}(n-2k)}}a_{k,I} \int_{M} \mathrm{ch}_{k}(\xi) \cdot p_{I}(g) = \sum_{k=0}^{l} \sum_{\substack{I\\\text{partitions of }\\ \frac{1}{4}(n-2k)}}b_{k,I} \int_{M} \mathrm{ch}_{k}(\xi)\cdot p_{I}(g)
\end{equation}
holds. 

Consider the manifold $\mathbb{C}P^{1}$ and the complex line bundle $\xi=T\mathbb{C}P^{1}$ with the usual hermitian structures induced by the Fubini-Study metric. Then the Chern character of $\xi$ is given by
$$
  \mathrm{ch}(\xi)= 1+ c_{1}(\xi),
$$
with 
\begin{equation*}
  \int_{\mathbb{C}P^{1}} c_{1}(\xi) = -2.
\end{equation*}
Since as a real manifold $\mathbb{C}P^{1}\cong S^{2}$ is two-dimensional, the total Pontryagin class of $\mathbb{C}P^{1}$ is trivial. For $0\leq j \leq l$ consider manifolds of the form $(\mathbb{C}P^{1})^{j}\times M_{n-2j}$, where $M_{n-2j}$ is a $n-2j$-dimensional manifold, equipped with the hermitian line bundle 
$$
  \xi^{\otimes j} = \pi_{1}^{*}\xi \otimes \dots \otimes \pi_{j}^{*}\xi \to (\mathbb{C}P^{1})^{j}\times M_{n-2j}.
$$ 
and some product metric $g$.
Here, $\pi_{i}\colon (\mathbb{C}P^{1})^{j}\times M\to \mathbb{C}P^{1}$ is the projection onto the $i$th component. For the Chern character 
\begin{equation*}
  \operatorname{ch} \xi^{\otimes j}= \prod_{i=1}^{j}(1+c_{i}) = \sum_{k=1}^{j}\sigma_{k}(c_{1},\dots,c_{j})
\end{equation*}
holds, where $c_{i}= \pi^{*}_{i}c_{1}(\xi)$. For the total Pontryagin classes we have 
\begin{equation*}
  p((\mathbb{C}P^{1})^{j}\times M_{n-2j}) = p((\mathbb{C}P^{1})^{j})p(M_{n-2j})= p(M_{n-2j }).
\end{equation*}
For a partition $I$ of an integer $k$, denote the $I$th Pontryagin number of a $4k$-dimensional manifold $N$ by
\begin{equation*}
  p_{I}[N]:= \langle p_{I}(N),[N]\rangle = \int_{N} p_{i_{1}}\wedge\dots \wedge p_{i_{r}}.
\end{equation*}

Let $I$ be a partition of $\frac{1}{4}(n-2k)$ for $k<j$, then
\begin{align*}
  \int_{(\mathbb{C}P^{1})^{j}\times M_{n-2j}} \mathrm{ch}_k(\xi^{\otimes j})\wedge p_{I}(g) = {j \choose k}\cdot \left(\int_{\mathbb{C}P^{1}} c_{1}(\xi)\right)^{k} \cdot \underbrace{p_{I}[\mathbb{C}P^{j-k}\times M_{n-2j}]}_{=0}= 0.
\end{align*}
For a partition $I$ of $\frac{1}{4}(n-2j)$, we have 
\begin{align*}
  \int_{(\mathbb{C}P^{1})^{j}\times M_{n-2j}} \mathrm{ch}_j(\xi^{\otimes j})\wedge p_{I}(g) = (-2)^{j}p_{I}[M_{n-2j}]
\end{align*}
and for a partition $I$ of $\frac{1}{4}(n-2k)$ for $k>j$, the $2k$-form part $\mathrm{ch}_{k}$ of the Chern character of $\xi^{\otimes j}$ is zero, and thus 
\begin{align*}
  \int_{(\mathbb{C}P^{1})^{j}\times M_{n-2j}} \mathrm{ch}_k(\xi^{\otimes j})\wedge p_{I}(g) = 0.
\end{align*}
Applying equation \ref{eq:equalityofforms} to the manifold $(\mathbb{C}P^{1})^{j}\times M_{n-2j}$ yields
\begin{equation}\label{eq:pontryaginnumbers}
  \sum_{\substack{I\\\text{partitions of }\\ \frac{1}{4}(n-2j)}}a_{j,I}  \cdot p_{I}[M_{n-2j}] =  \sum_{\substack{I\\\text{partitions of }\\ \frac{1}{4}(n-2j)}}b_{j,I} \cdot p_{I}[M_{n-2j}].
\end{equation}
Note that $j\in \{ 0,\dots,l \}$ and $M_{n-2j}\in \mathbf{Man}_{n-2j}^{G_{n-2j}}$ were chosen arbitrarily. 
To deduce that $a_{j,I}$ and $b_{j,I}$ coincide we invoke the following Theorem of Thom.

\begin{thm}\cite[Theorem 16.8]{milnorCharacteristicClasses1974}
	Let $s_i$ be the power sum symmetric polynomial, and $M_i$ be a sequence of manifolds of dimension $4i$ with $s_i(p)[M_i]\neq 0$. Then the matrix 
	\[(p_I[M_{j_1}\times\dots\times M_{j_r}])_{I,J\text{ partitions of }k}\]
	is non-singular.
\end{thm}

Let $M_{1}$ be the $K 3$-surface and $M_{j}= \mathbb{H}P^{j}$, $j\geq 2$ be the quaternionic projective space. These manifolds are all spinnable, hence carry $G_{n}$-structures, and have the property that $s_{i}(p)[M_{i}]\neq 0$ \cite[2.3]{atiyahSpinManifoldsGroupActions1970}. Applying equation \ref{eq:pontryaginnumbers} to the manifolds $M_{i_{1}}\times\dots \times M_{i_{r}}$ yields the equality of the coefficients $a_{j,I}$, $b_{j,I}$. 

Renaming the differential form \ref{eq:difffrom1} to
\begin{equation*}
  (-1)^{l} \left(\mathrm{ch}(\nabla^{\xi}) \cdot \frac{\mathrm{ch}(V_{+})-\mathrm{ch}(V_{-})}{\chi}(\nabla^{g})\cdot \hat{A}(\nabla^{g})^{2}\right)_{n},
\end{equation*}
we have just shown the following result.

\begin{thm}\label{thm:localindextheorem}
	Let $(\sigma,\varepsilon)$ be a chiral $G_n$-geometric symbol for $n=2l$ even. Let $(M,P)$ be a Riemannian $G_n$-manifold equipped with an hermitian vector bundle $\xi$. Let $D^{\xi}$ be the induced twisted geometric operator. Then the equality
	\begin{equation*}
		\lim_{t\searrow 0}\left(\operatorname{str}\exp(-t (D^{\xi})^{2})\cdot\mathrm{dvol}_g\right) = (-1)^{l} \left(\mathrm{ch}(\nabla^{\xi}) \cdot \frac{\mathrm{ch}(V_{+})-\mathrm{ch}(V_{-})}{\chi}(\nabla^{g})\cdot \hat{A}(\nabla^{g})^{2}\right)_{n}
	\end{equation*}
	holds.
\end{thm}

Applying Theorem \ref{thm:localindextheorem} to the Rarita--Schwinger, the higher Dirac operators and the higher signature operators and using Equations \ref{eq:indexrs} and \ref{eq:indexhigherdirac}, we obtain the corresponding local index theorems.
\begin{cor}\label{cor:RS}
	Let $Q^{\xi}$ be the Rarita--Schwinger operator on an even-dimensional Riemannian spin-manifold $(M,P)$, twisted with some hermitian vector bundle $\xi$. Then
	\begin{equation*}
		\lim_{t\searrow 0}\left(\operatorname{str}\exp(-t(Q^{\xi})^{2})\cdot\mathrm{dvol}_g\right) = \left(\mathrm{ch}(\nabla^{\xi})\left(\mathrm{ch}(T_\mathbb{C})(\nabla^{g})+1\right)\hat{A}(\nabla^{g})\right)_n.
	\end{equation*}
\end{cor}

\begin{cor}\label{cor:higherD}
	Let \(D_j^{\xi}\) be the higher Dirac operator on an even dimensional Riemannian spin-manifold \((M,P)\), twisted with some hermitian vector bundle $\xi$. Then
  \begin{align*}
    \lim_{t\searrow 0} \left(\operatorname{str}\exp(-t(D^{\xi}_{j})^{2})\mathrm{dvol}_g\right) = \left(\mathrm{ch}(\nabla^{\xi})\cdot\left(\mathrm{ch}(\Lambda^{j}T_{\mathbb{C}})+\mathrm{ch}(\Lambda^{j-1}T_{\mathbb{C}})\right)(\nabla^{g})\cdot \hat{A}(\nabla^{g})\right)_{n}.
  \end{align*}
\end{cor}

\begin{cor}\label{cor:higherP}
	Let $P_{\mu}^{\xi}$ be the higher signature operator on an oriented 4-dimensional Riemannian manifold $(M,g)$, twisted with some hermitian vector bundle $\xi$. Then
	$$
    \lim_{t\searrow 0} \left(\operatorname{str}\exp(-t(P^{\xi}_{\mu})^{2})\mathrm{dvol}_g\right) = \left( 1+\mu \right)\left(\mathrm{ch}(\nabla^{\xi})\cdot L(\nabla^{g})\right)_{n}.
	$$
\end{cor}

\backmatter

\bmhead{Acknowledgements}

I would like to thank my Ph.D. advisor Christian Bär for proposing this project as part of my Ph.D. thesis. I am especially indebted to Mehran Seyedhosseini for many helpful discussions.


\bibliography{bibliothek}

\end{document}